\documentclass[a4paper,10pt]{article}

\usepackage{amsmath,amsthm,amsfonts,amssymb,euscript,dsfont,mathrsfs,enumerate}
\usepackage{xcolor}
\usepackage[english]{babel}
\usepackage[utf8]{inputenc}

\usepackage{subcaption}
\usepackage{algorithm}
\usepackage{algorithmic}
\usepackage{graphicx}

\newcommand{\attractor}[1]{}

\newtheoremstyle{nopunct}%
  {3pt}
  {3pt}
  {\itshape}
  {}
  {\bfseries}
  {}
  { }
  {}

\DeclareUnicodeCharacter{00A0}{~}
\usepackage{nameref}

\newtheoremstyle{questionstyle}
  {\topsep}{\topsep}      
  {\itshape}              
  {}                       
  {\bfseries}              
  {.}                      
  {.5em}                   
  {\centering #1~#2}       

\newtheorem{assumption}{Assumption}
\newtheorem{prop}{Proposition}

\newtheorem{theorem}{Theorem}
\newtheorem{coro}{Corollary}
\newtheorem{lemma}{Lemma}

\theoremstyle{questionstyle}

\theoremstyle{nopunct}
\newtheorem*{theorem*}{Theorem}
 
\newtheorem*{result*}{Result.}
\newtheorem*{question*}{Question.}

 \usepackage{authblk}
\usepackage{geometry}
\newcommand{\un}{\mathds{1}}
\newcommand{\ps}[1]{\langle #1\rangle}
\newcommand{\bP}{{\mathbb P}}
\newcommand{\bE}{{\mathbb E}}
\newcommand{\bR}{{\mathbb R}}

\newcommand{\xrs}{ x^{\rm RS} }

\newcommand{\cF}{{\mathcal F}}

\newcommand{\dr}{{d}}
\newcommand{\ex}{\mathrm{e}}

\newcommand\esp[1]{\mathbb{E}\left[#1\right]}

\newcommand\p[1]{\left( #1 \right)}
\newcommand{\espcond}[2]{\mathbb{E}\mathopen{}\left[#1\middle|#2\right]}

\newcommand\norm[1]{\left\lVert #1 \right\rVert}
\newcommand\abs[1]{\left\lvert #1 \right\rvert}

\usepackage{authblk}

\definecolor{darkgreen}{RGB}{0,128,60}  

\usepackage{pdfpages}
\usepackage{etoc}
\author[1]{{Radu-Alexandru}~{Dragomir}}

\author[2]{{François}~{Portier}}
\author[1]{{Victor}~{Priser}}

\affil[1]{LTCI, T\'el\'ecom Paris}
\affil[2]{CREST, ENSAI}

\title{Importance Sampling Optimization with Laplace Principle}

\usepackage{hyperref}
\begin{document}

\maketitle
\begin{abstract}
    Grid search and random search are widely used techniques for hyperparameter tuning in machine learning, especially when gradient information is unavailable. In these methods, a finite set of candidate configurations is evaluated, and the best-performing one is selected. We propose a simple and computationally inexpensive refinement of this paradigm: instead of selecting a single best point, we form a weighted average of the evaluated configurations, where the weights are chosen using an importance sampling scheme inspired by the Laplace principle. This scheme can be implemented as a post-processing step on top of a random search, with no additional function evaluations.
    We also propose an iterative variant, where the sampling distributions are chosen adaptively to generate new candidate points around the previous estimate, in the spirit of Evolution Strategy (ES) methods. 
    
In a general non-convex setting, we show that, after $n$ evaluations, the error of the proposed methods is of smaller order than $n^{-2/(d+2)} $. This compares favorably to random search or grid search rates of $n^{-1/d}$ as soon as $d> 2$. 
We illustrate the practical benefits of this averaging strategy on several examples.

\end{abstract}

\section{Introduction}
\label{sec:Introduction}


We consider the global optimization problem
$$\min_{x\in \mathbb R^d}  f(x),$$
where $f$ is a possibly nonconvex function with a unique minimizer $x^*$. We assume that the gradients are unknown or difficult to compute, motivating the use of gradient-free schemes. Applications of interest include hyperparameter optimization in machine learning \cite{koch2018autotune,feurer2019hyperparameter}, policy search in reinforcement learning \cite{stulp2012path,salimans2017evolution}, and ODE parameter estimation \cite{Villaverde2018}.

 A large number of global optimization methods involve exploring the space by sampling candidate points at random. The simplest one is plain random search, which draws $n$ i.i.d. samples $X^1,\dots, X^n$ from a distribution $q_0$ and outputs the best solution
\begin{equation}\label{eq:rs_intro}  \xrs_n = \arg \min \{ f(x) \,:\, x \in \{ X^1,\dots, X^n \} \}.
\end{equation}
More sophisticated algorithms, such as \textit{Evolution Strategies} \cite{Beyer2002}, perform random search in several stages. After each stage, a subset of the best performing solutions is retained and used to adapt the sampling policy for the next iteration. A notable example is CMA-ES \cite{hansen2006cma}, which is among the state-of-the-art algorithms for global optimization \cite{loshchilov2016cma}.

 \paragraph{Contributions.} In this work, we propose a scheme  to improve the performance of random search-based methods. Instead of taking the best point in \eqref{eq:rs_intro}, we suggest to compute a weighted average
\begin{equation}\label{eq:liso_intro}
    x_n =  \frac{\sum_{i=1}^n w^i X^i}{\sum_{i=1}^n w^i } \; \;\;\text{where } \;\; w^i = \frac{\exp(-\alpha f(X^i))}{ q_0(X^i) },
\end{equation}
and $\alpha > 0$ is a parameter. 
We show that, when $f$ is sufficiently smooth and locally convex around $x^*$, this approach achieves better guarantees than selecting the best sample $\xrs$ in~\eqref{eq:rs_intro} when $d > 2$. 
Building on this idea, we propose an iterative variant in which the sampling distribution is adapted at each stage to refine the search (for instance, using a Gaussian centered at the previous estimate). 

We demonstrate the performance of our approach on several global optimization benchmarks, by comparing it to other random search strategies.

\paragraph{Key idea: Laplace importance sampling.} 
Equation \eqref{eq:liso_intro} is motivated by a sampling scheme applied to the distribution 
\[
\pi_\alpha(x) := \frac{1}{Z_\alpha} \exp(-\alpha f(x)), \,{Z_\alpha} := \int \exp(-\alpha f(u)) du.
\]
For large $\alpha$, $\pi_\alpha$ is concentrated around the minimizer of~$f$, and its expectation provides a good approximation of $x^*$: this is known as the \textit{Laplace principle}. As sampling directly from $\pi_\alpha$ is untractable, we tackle the problem by using importance sampling (IS) \cite{robert1999monte,agapiou2017importance}. This classical sampling technique consists in drawing points from a simpler distribution $q_0$, and combining them with appropriate weights to estimate $\mathbb{E}_{X\sim\pi_\alpha}[X]$.

We choose the parameter $\alpha$ as to balance the variance of the estimator and the bias term $\|\mathbb{E}_{X\sim\pi_\alpha}[X] - x^*\|$. In our theory, the optimal scaling is $\alpha \propto n^{2/(d+2)}$.

Note that \eqref{eq:liso_intro} can also be interpreted as a smoothed version of the argmin operator (or \textit{softmin}), with an additional correction by $q_0(X^i)$ due to importance sampling.

\paragraph{Related work.} 
The idea of performing optimization by sampling from $\pi_\alpha \propto e^{-\alpha f}$ has been explored through several approaches, including simulated annealing based on Langevin dynamics \cite{Chiang197SA,xu2018global} or Metropolis-Hasting schemes \cite{kirkpatrick1983optimization}. The main novelty of our work, in contrast to these methods, is the use of importance sampling on $\pi_\alpha$.

The Laplace principle \cite{hwang1980laplace}, and its quantitative version from \cite{kirwin2010higherasymptoticslaplacesapproximation}, has also been used in the theoretical analysis of particle swarm/consensus-based optimization, to smooth the argmin operator \cite{pinnau2017consensus,Fornasier_2024,bianchi2025consensus}.

Our adaptive algorithm is inspired by adaptive importance sampling methods \cite{ho+b:1992,bugallo2017adaptive,portier2018asymptotic}. It is also closely related to the class of Evolution Strategy methods \cite{Beyer2002,hansen2006cma}. While the latter also rely on weighted averaging, the weights are determined by ranking-based coefficients \cite{ARNOLD200618}, rather than our Laplace importance weights.
Finally, our approach shares some similarities with \textit{cross-entropy methods} \cite{de2005tutorial}, who also rely on importance sampling for minimizing $f$. These methods differ in that they focus on fitting a probability model to estimate the distribution of the ``elite samples'' instead of using the measure $\pi_\alpha$.

    \section{Algorithms}
    \label{sec:algo}

The basic scheme, called \textit{Laplace Importance Sampling Optimization} (LISO), is described in Algorithm \ref{algo:IS} below. It samples $n$ points according to a distribution $q_0$ and outputs a weighted average as in \eqref{eq:liso_intro}.

\begin{algorithm}[ht]
\caption{Laplace Importance Sampling Optimization (LISO)}
\label{algo:IS}
\begin{algorithmic}
\setlength{\baselineskip}{1.6\baselineskip} 
\STATE{\textbf{Input:} $\alpha > 0$, sample number $n$, distribution $q_0$}

\STATE{ Sample $X^1,\dots, X^n \sim_{i.i.d.} q_0$}

\STATE{ Compute weights $w^i = \frac{\exp(-\alpha f(X^i))}{ q_0(X^i) }$}

\STATE{\textbf{Output:} $x_n= \frac{\sum_{i=1}^n w^i X^i }{ \sum_{i=1}^n w^i  }$}
\end{algorithmic}
\end{algorithm}

Algorithm \ref{algo:AIS} presents an adaptive variant of LISO in which the sampling distribution $q_i$ is adjusted at each iteration. Specifically, at iteration $i$, the method computes a Laplace weighted average $\mu_i$ based on the samples $X^1,\dots, X^i$. Then, the next point $X^{i+1}$ is drawn from the distribution $q_i$ chosen as the mixture
 \[
    q_i = (1-\lambda)\mathcal N(\mu_i,\sigma^2I_d) + \lambda q_0.
 \]
 Using the Gaussian centered at $\mu_i$ promotes local exploitation around the previous estimate, while the mixture with $q_0$ allows to guarantee sufficient exploration of the search space.
 Note that other choices of adaptive sampling policies beyond Gaussian updates are possible; see Section \ref{ss:aliso}.
 
\begin{algorithm}[H]
\caption{Adaptive LISO}
\label{algo:AIS}
\begin{algorithmic}
\setlength{\baselineskip}
{1.6\baselineskip} 
\STATE{\textbf{Input:} $n$, $q_0$, $(\alpha_n)$, $\lambda$,$\sigma$.}

\STATE{ Initialize $X^1 \sim q_0$}

\FOR{ $i=1,2 \dots, n-1$}

\STATE{ $\mu_i = \frac{ \sum_{k=1}^i w^{k}_i X^k }{ \sum_{k=1}^i w^{i,k} }$, where $w^{i,k} = \frac{\exp(-\alpha_i f(X^k))}{q_{k-1}(X^k)}$}
    
\STATE{ $q_i = (1-\lambda)\mathcal N(\mu_i,\sigma^2I_d) + \lambda q_0$}
    
 \STATE{ Sample $X^{i+1} \sim q_i$}
    
\ENDFOR

\STATE{\textbf{Output:} $x_n =\mu_n$}
\end{algorithmic}
\end{algorithm}

\paragraph{Complexity.}
    Both Algorithms \ref{algo:IS} and \ref{algo:AIS} require $n$ evaluations to $f$. Algorithm \ref{algo:IS} can be run in parallel, while Algorithm \ref{algo:AIS} must be executed sequentially, and performs a different weighted combination at each iteration. As a result, it requires $\mathcal{O}(dn^2 + n C_{f})$ operations, where $C_f$ is the cost of evaluating $f$, compared to $\mathcal{O}(dn + nC_f)$ for the standard variant. However, in many applications the evaluation of~$f$ dominates the computational cost, in which case this overhead of Algorithm \ref{algo:AIS} is negligible.


\section{Theoretical analysis}

\subsection{Informal statement of main result}\label{ss:informal}

Throughout this section, we make the following standing assumption on the objective function.

\begin{assumption}\label{hyp:f}
The function $f$ is four-times continuously differentiable with a unique global minimizer $x^\ast$. The Hessian $H_f(x^\ast)$ is positive-definite and we have $\int \exp( - f(x) ) \, \dr x<\infty$.
\end{assumption}

Let us first give an informal statement of the convergence rate of the LISO method (Corollary \ref{cor:LISO}).

\begin{theorem}[LISO, informal]\label{th:informal}
Let $f$ and $q_0$ satisfy Assumption~\ref{hyp:f} as well as certain integrability conditions at infinity. Then, for $\alpha,n$ large enough, the output $x_n$ of Algorithm~\ref{algo:IS} after $n$ iterations satisfies
    \begin{equation*}
        \mathbb{E}\bigl\|x_n - x^\ast\bigr\|^2
    \le  \frac{C}{n} \alpha^{\frac{d}{2}-1}  + \frac{C'}{\alpha^2},
    \end{equation*}
    where $C,C'$ are constants depending on $f$ and $q_0$.
\end{theorem}
By minimizing the bound with respect to $\alpha$, we find that the theoretically optimal choice is $\alpha \propto n^{2/(d+2)}$, for which the error estimate of 
\(
\mathbb{E}\bigl\|x_n - x^\ast\bigr\|^2
\)
is of order $\mathcal{O}\bigl(n^{-4/(d+2)}\bigr)$. This shows that, with this specific choice of $\alpha$, our LISO algorithm possesses asymptotically better performance than the random search algorithm when $d \ge 3$, since the output $x_n^{\mathrm{RS}}$ of the random search algorithm has an error estimate
\(
\mathbb{E}\bigl\|x_n^{\mathrm{RS}} - x^\ast\bigr\|^2 = \mathcal{O}\bigl(n^{-2/d}\bigr)
\).


\subsection{Laplace approximation}

Let us introduce the fundamental Laplace principle underlying our analysis.
We will consider integrals of the form
\[
\int \exp(-\alpha f(x)) \, \varphi(x)\, dx,
\]
for some function \( \varphi : \mathbb{R}^d \to \mathbb{R}^q \).
The assumption required for this function \( \varphi := (\varphi_i)_{i\in[q]} \) is given below. Define
\begin{align*}
    C_{f,\varphi} := 
    \sum_{i\in[q]}\abs{ \mathrm{Tr}( H_{\varphi_i}(x^\ast)) - \ps{\nabla\varphi_i(x^\ast),\mathrm{PTr}(D^3 f(x^\ast))}},
    \end{align*}    
where $H_{\varphi_i}$ denotes the Hessian of $\varphi_i$, $\mathrm{Tr}$ denotes the trace operator, and the partial trace of the tensor $D^3 f$ is defined as
\(
\mathrm{PTr}(D^3 f(x^\ast)) := \sum_{k \in [d]} \bigl(\partial_{ikk} f(x^\ast)\bigr)_{i \in [d]}
\).

\begin{assumption}\label{hyp:varphi}
The function $\varphi$ is twice continuously differentiable and $C_{f,\varphi}>0$. Moreover $\varphi $ satisfies the integrability condition:
$$
\int \exp (- f(x) ) \|\varphi(x)\| \,\dr x < \infty.
$$  
\end{assumption}

Since our result will hold for large values of $\alpha$, the assumption above can be weakened to the Lebesgue-integrability of  $ \exp (- \alpha_0 f(x )) \|\varphi(x)\| $ for some $\alpha_0 >1$. While this would allow Assumption~\ref{hyp:varphi} to hold on a slightly larger set of functions $f$ and $\varphi$, this would diminish clarity of the statement.

Recall that $\pi_\alpha$ is the probability distribution with density proportional to $\exp\bigl(-\alpha f\bigr)$. Then, both Assumptions~\ref{hyp:f} and~\ref{hyp:varphi} ensure that the quantity $\pi_\alpha(\varphi)$ is well-defined for $\alpha \ge 1$ and, in fact, this quantity approaches $\varphi(x^\ast)$ as $\alpha \to \infty$.  
This result, known as the \emph{Laplace principle}, is formalized below in the lemma below. It also establishes that the unnormalized quantity
\(
\int \exp(-\alpha f(x)) \, \varphi(x) \, dx
\)
is of order $\alpha^{-d/2} \exp\bigl(-\alpha f(x^\ast)\bigr)$.
\begin{lemma}[Laplace principle from \cite{kirwin2010higherasymptoticslaplacesapproximation}, Theorem~1.1]
\label{lemma:laplace_approx}
Let Assumptions~\ref{hyp:f} and~\ref{hyp:varphi}  hold true. 
For every $\alpha>1$, we obtain
\[
    \norm{ m_\alpha \int \ex^{-\alpha (f(y)- f(x^\ast))} \varphi(y) \,\mathrm d y - \varphi(x^\ast)} \le    \frac {C }{\alpha},
\]
with 
\begin{equation}\label{eq:mlpha}
 m_\alpha := \sqrt{\det H_f(x^\ast)}\p{\frac{\alpha}{2\pi}}^{d/2}\,,   
\end{equation}
and the constant $C $ depends on $f,\varphi$ only. 
Moreover, there exists $\tilde \alpha>1$ depending on $f,\varphi$ only, such that for $\alpha \geq \tilde \alpha$:
$$\norm{ \pi_\alpha( \varphi) - \varphi(x^\ast)}\le   \frac {C_{f,\varphi} } \alpha\,. $$ 
\end{lemma}

We remark that, in Lemma~\ref{lemma:laplace_approx} above, the constant $C_{f,\varphi}$ in the second statement is made explicit, whereas the constant $C$ in the first statement is not. Although this could be done using \cite {kirwin2010higherasymptoticslaplacesapproximation}, we choose not to do so in order to preserve simplicity, since in our main theorems the constant $C$ is involved in higher order terms than $C_{f,\varphi}$ in the upper bound on the error.

Note that Lemma~\ref{lemma:laplace_approx} above holds under the condition $C_{f,\varphi} > 0$. The constant $C_{f,\varphi}$ vanishes, for instance, when $\varphi$ is linear and $f$ is quadratic. However, in this case, using a classical result on Gaussian vectors with covariance matrix $\alpha^{-1} I_d$, Lemma~\ref{lemma:laplace_approx} still holds.  
A situation where Lemma~\ref{lemma:laplace_approx} would not hold is when $\varphi$ is linear and $f$ is quadratic on an open set around $x^\ast$ but not outside this set. In this case, the error $\| \pi_\alpha(\varphi) - \varphi(x^\ast) \|$ would not be zero, as Lemma~\ref{lemma:laplace_approx} implies, but rather a quantity of higher order in $\alpha^{-1}$. For simplicity, we have excluded these particular cases from our study.

\subsection{Analysis of the basic LISO scheme}

Recall that $q_0: \mathbb R^d\to \mathbb R_{\geq 0} $ is the probability density function  used to sample the points.  Let $(X^{i})_{i \in [n]} $ be an independent collection of random variables with common distribution $ q_0$. The importance weights are defined as
    $$w^{i} := \frac{\exp(-\alpha f(X^{i}))}{q_0(X^{i})}.$$
The LISO estimator is 
\begin{align*}
&\pi_{n,\alpha}(\varphi)  := \frac{\sum_{i \in [n]} w^{i} \, \varphi  (X^{i})} {\sum_{i \in [n]} w^{i}}.
\end{align*}
LISO is part of the family of \textit{self-normalized importance sampling estimators} \cite{agapiou2017importance,portier2018asymptotic}. It satisfies the following invariance property: the unknown weights  
\[
w^{\ast i}  = \frac{\exp\bigl(-\alpha (f(X^{i}) - f(x^\ast))\bigr)}{q_0(X^{i})}
\] 
can be used in place of $w^i$ without changing the estimator $\pi_\alpha(\varphi)$. It means that
\begin{align}\label{self_norm}
\pi_{n,\alpha}(\varphi) = \frac{\sum_{i \in [n]} w^{\ast i} \, \varphi  (X^{i})} {\sum_{i \in [n]} w^{\ast i}}.    
\end{align}

This straightforward property turns out to be useful in our mathematical development, as $f - f(x^\ast)\ge 0$.
The next assumption states that the function $q_0$ dominates the measure $\pi_\alpha$. By Equation~\eqref{self_norm}, it can be formulated in terms of the nonnegative function \(f(x) - f(x^\ast)\) instead of \(f\).

\begin{assumption}\label{hyp:q0}
The function \(q_0\) is a twice continuously differentiable density supported on \(\mathbb{R}^d\), and the tail of \(q_0\) is heavy compared to that of \(\exp(-(f - f(x^\ast)))\), i.e., there exists \(c_0>0\) such that
        $$\forall x\in \mathbb R^d ,\qquad \exp (- ( f(x) - f(x^\ast)  )  \leq c_0 q_0(x)  .$$
       Moreover, $ \varphi  $ satisfies the following integrability conditions:
        \begin{align*}
       & \int \exp (- ( f(y) - f(x^\ast))) \|\varphi(y)\|^2\, \dr y< \infty \,,\\
        &    \int  \|\varphi(y)\|^4 q_0 (x) \, \dr y< \infty \,.
        \end{align*}
        In addition, we assume that  $\| \varphi \| ^2/q_0 $, $\| \varphi \| ^2/q_0$, $1/q_0$ are twice-continuously differentiable, and  $C_{f, \|\bar \varphi \| ^2/q_0 } >0 $, $C_{f, \|\bar \varphi \| ^2/q_0 }>0$, $C_{f,1/q_0}$, where we define \(\bar\varphi = \varphi - \varphi(x^\ast)\).
\end{assumption}

Similarly to Assumption~\ref{hyp:varphi}, one could weaken these assumptions by requiring the existence of \(\alpha_0 > 1\) such that the above conditions hold with \(\exp(-\alpha_0 (f - f(x^\ast)))\). Note also that, by Jensen's inequality, the moment condition involving \(\|\varphi\|^2\) stated above is stronger than the moment condition involving \(\|\varphi\|\) in Assumption~\ref{hyp:varphi}.

Our main theorem on LISO is obtained by controlling the variance of the importance sampling estimator $\pi_{n,\alpha}(\varphi)$. To this end, we need to control the variance of the importance weights $w^{\ast i}$. This justifies Assumption~\ref{hyp:q0}, since it ensures that this variance is bounded:
\[
\int \exp\bigl(-2 ( f(y) - f(x^\ast))\bigr)\,\frac{\|\varphi(y)\|^2}{q_0(y)} \,dy < \infty .
\]
Moreover, the additional requirements of Assumption~\ref{hyp:q0} allow us to apply the Laplace principle (Lemma~\ref{lemma:laplace_approx}) to the function \(\|\varphi\|^2 / q_0\), thereby enabling a better control of the variance of the weights.

Let us now state our first main result - an upper bound on the mean squared error of the LISO method, which is proven in Section~\ref{sec:proofTh1} of the appendix.

\begin{theorem}
\label{th:1}
    
Let Assumption~\ref{hyp:f} and \ref{hyp:q0} be fulfilled. There exist a constant \( c>0\) and a constant \(C>0\), depending on \(f,q_0,\varphi\), a constant \(\tilde \alpha>0\), depending on \(f,\varphi\), such that for every \(\alpha>\tilde \alpha\) and every \(n\) satisfying
\[
n \geq c {\alpha^{d/2}\log\!\bigl(n{\alpha^{-d/2+1}}(\sqrt n +\alpha^{d/4})\alpha\bigr)} ,
\]
we have, with $\bar \varphi:=\varphi -\varphi(x^\ast)$,
\begin{equation*}
\mathbb{E} \bigl\|\pi_{n,\alpha}(\varphi)-\pi_{\alpha}(\varphi)\bigr\|^2
\le
9\,n^{-1} m_\alpha\,2^{-d/2}
\frac{C_{ f, \|\bar\varphi\|^2/q_0}}{2\alpha}
\left(1+\frac{C}{\alpha}\right).
\end{equation*}
\end{theorem}
We decompose 
            $$\bE\norm{\pi_{n,\alpha}(\varphi) - \varphi(x^\ast)}^2\le \bE\norm{\pi_{n,\alpha}(\varphi) - \pi_{\alpha}(\varphi)}^2 + \norm{\pi_{\alpha}(\varphi) - \varphi(x^\ast)}^2 . $$ 
            Then applying Theorem~\ref{th:1} and Lemma~\ref{lemma:laplace_approx}, respectively to each term, we obtain the following corollary.
            
\begin{coro}\label{cor:LISO} 
Under the assumptions of Theorem~\ref{th:1}, with Assumption~\ref{hyp:varphi} additionally required, for every $\alpha,n$ satisfying the conditions of Theorem~\ref{th:1}, we obtain
\begin{equation*}
    \mathbb{E}\bigl\|\pi_{n,\alpha}(\varphi)-\varphi(x^\ast)\bigr\|^2
\le 9\,n^{-1} m_\alpha\,2^{-d/2}
\frac{C_{ f, \|\bar\varphi\|^2/q_0}}{2\alpha}
\left(1+\frac{C}{\alpha}\right) +\frac{\bigl(C_{f,\varphi}\bigr)^2}{\alpha^{2}}\,.
\end{equation*}

\end{coro}
This last corollary requires the additional Assumption~\ref{hyp:varphi}, since it relies on the use of Lemma~\ref{lemma:laplace_approx} applied to \(\varphi\), which was not needed in the proof of Theorem~\ref{hyp:f}. We note that Corollary~\ref{cor:LISO} allows us to deduce the informal Theorem~\ref{th:informal} with $\varphi =\mathrm{id}$ where $\mathrm{id} : x\mapsto x$.

\paragraph{Proof ideas}
\label{sec:sketchTh1}

Since the random variables
\(
\frac{w^i \varphi(X^i)}{\sum_{i\in[n]} w^i}
\)
are not independent, standard theory (law of large number, Bernstein inequality...) cannot be applied directly in the proof of Theorem~\ref{th:1}.


Then, we first study the \emph{unnormalized version} of $\pi_{n,\alpha}(\varphi)$:
\begin{align*}
    \tilde \pi_{n,\alpha}(\varphi) = \frac{1}{n Z_\alpha} \sum_{i \in [n]} w^i \, \varphi(X^i),
\end{align*}
where
\(
Z_\alpha := \int \exp(-\alpha f)
\).
In this case, $\tilde \pi_{n,\alpha}(\varphi)$ is the average over $n$  i.i.d. random variables $Z_\alpha^{-1}w^{ i} \varphi^i$, which have expectation $\pi_\alpha(\varphi)$ and bounded variance, by Assumption~\ref{hyp:q0}.
Then, we obtain the following result, which is proven in Section~\ref{sec:ProofProp}.
\begin{prop}\label{prop:unnormalized}
Suppose that Assumptions~\ref{hyp:f} and \ref{hyp:q0} are fulfilled. 
There exists a constant \(\tilde \alpha>1\), depending on \(f\) and \(\varphi\), such that, 
for every \(\alpha>\tilde \alpha\), 
\begin{align*}
&\mathbb{E}\!\left[\bigl\|\tilde \pi_{n,\alpha}(\varphi)-\pi_{\alpha}(\varphi)\bigr\|^2\right]\le n^{-1} 2^{-d/2}  m_{\alpha}  \left( \frac{\norm{\varphi(x^\ast)}^2}{q_0(x^\ast)} + \frac{C_{f,\norm{\varphi} ^2 /q_0 } }{2 \alpha} \right)\left( 1+   \frac C \alpha   \right),    
\end{align*}
where $C$ depends only on $f$ and $C_{f,\varphi}$ is the constant of Lemma \ref{lemma:laplace_approx}. 
\end{prop}
Then, the connection between \(\tilde\pi_{\alpha}(\varphi)\) and \(\pi_{n,\alpha}(\varphi)\) follows from the observation that 
\(\pi_{n,\alpha}(\varphi) = \tilde \pi_{n,\alpha}(\varphi)/\tilde \pi_{n,\alpha}(1)\). By the result in Proposition~\ref{prop:unnormalized}, 
we have \(\tilde \pi_{n,\alpha}(1) \to \pi_\alpha(1) = 1\) and \(\tilde \pi_{n,\alpha}(\varphi) \to \pi_{n,\alpha}(\varphi) \).

We remark that in Proposition~\ref{prop:unnormalized}, the bound is of order \(m_\alpha n^{-1}\), whereas in Theorem~\ref{th:1} 
the bound is of order \(\alpha^{-1} m_\alpha n^{-1}\). This higher-order bound is obtained by noting that, since \(\pi_\alpha\) and 
\(\pi_{n,\alpha}\) are probability measures, we can write (with \(\bar\varphi = \varphi - \varphi(x^\ast)\))
\[
\pi_{n,\alpha}(\varphi) - \pi_{\alpha}(\varphi) = \pi_{n,\alpha}(\bar\varphi) - \pi_{\alpha}(\bar\varphi).
\]
This allows us to work with \(\bar\varphi\) instead of \(\varphi\). And, since \(\bar\varphi(x^\ast) = 0\), the bound 
in Proposition~\ref{prop:unnormalized} becomes of order \(\alpha^{-1} m_\alpha n^{-1}\). This achieves the proof of Theorem~\ref{th:1}.

\subsection{Analysis of adaptive LISO}\label{ss:aliso}

\subsubsection{Motivation of the adaptive variant}

We note that the performance of Algorithm~\ref{algo:IS} relies strongly on the choice of the initial distribution $q_0$. A good choice of $q_0$ would naturally be a distribution that samples with high probability around $x^\ast$.  
This observation appears in Theorem~\ref{th:1} through the constant $C_{f,\|\bar{\varphi}\|^2/q_0}$, which decreases as $q_0(x^\ast)$ increases.
In particular, in high dimensions, if the distribution $q_0$ does not concentrate around $x^\ast$, then $q_0(x^\ast)$ is small and deteriorates the bound in Theorem~\ref{th:1}. Since $x^\ast$ is not known, the only way to improve the performance of Algorithm~\ref{algo:IS} is to propose an adaptive strategy that samples new particles $X^{k+1}$ according to a distribution $q_k$ that depends on all previous particles $X^1,\dots,X^k$, chosen in the hope that $q_k(x^\ast)$ increases as $k$ grows.
While adopting such an adaptive strategy is likely to improve empirical performance, providing a rigorous theoretical justification is not straightforward, particularly when attempting to establish superiority over the non-adaptive strategy with fixed $q_0$.

Beyond the adaptive Algorithm~\ref{algo:AIS} with Gaussian proposals described in Section~\ref{sec:algo}, we consider here a slightly more general setting in which the sampling distribution \(q_n\) is chosen from a general family, not necessarily Gaussian. We assume that, for every $n$, $q_n \in \mathcal{Q}$, where $\mathcal{Q}$ is a family of distributions parametrized by their mean:
\[
\mathcal{Q} := \{ q_\theta\, : \,\theta \in \Theta \},
\]
with $\Theta \subset \mathbb{R}^d$ denoting the parameter space, and each $\theta \in \Theta$ is the mean of $q_\theta$, that is,
\(
\theta = \int x \, q_\theta(x) \, \mathrm{d}x.
\)
The sequence of sampling distributions \((q_n)_{n \ge 0}\) is referred to as the sampling policy and may depend on the previously generated particles \((X^1, \ldots, X^n)\).
The adaptive LISO algorithm can be described as follows. Starting from an initial distribution \(q_0\), at each step \(n \ge 1\), we sample \(X^n\) from \(q_{n-1}\) and compute the corresponding weights (using notation similar to that introduced earlier).
\[
w^{i} := \frac{\exp(-\alpha f(X^{i})) }{q_{i-1} (X^{i})} \qquad \forall i = 1,\ldots, n \,.
\]
In the same way as we defined the LISO estimator $\pi_{n,\alpha}(\varphi)$, we define the adaptive LISO estimator as
$$\pi_{n,\alpha} ^{(A)}(\varphi)  := \frac{\sum_{i \in [n]} w^{i} \, \varphi  (X^{i})} {\sum_{i \in [n]} w^{i}}. $$

Given this, we need to choose the next sampling probability measure \(q_n =: q_{\theta_n} \in \mathcal{Q}\). As mentioned above, we want the measure \(q_n\) to sample around \(x^\ast\) with high probability.  
If we assume that the probability measures in \(\mathcal{Q}\) concentrate around their means, the optimal choice of $\theta_n$ would be \(x^\ast\). However, since this quantity is unknown, the best available estimate is the current approximation of \(x^\ast\), which—following the result of Corollary~\ref{cor:LISO} with $\varphi=\mathrm{id}$—is given by \(\pi_{n,\alpha}^{\mathrm{(A)}}(\mathrm{id})\). Then, we choose $q_n =q_{\theta_n} $ where $\theta_n =\pi_{n,\alpha}^{\mathrm{(A)}}(\mathrm{id}) $.

To ensure an exhaustive exploration of the domain, we propose to update \(q_n\) using the following mixture:
\begin{equation*}\label{eq:mixture}
    q_n = (1-\lambda) q_{\theta_n} + \lambda q_0 \,.
\end{equation*}

This strategy—by updating only the mean and thereby avoiding variance reduction—provides additional robustness: even if \(\theta_n\) drifts away from the minimizer, the sampling distribution retains sufficient spread to safely generate samples everywhere. Moreover, with small probability, we also allow sampling from a fixed distribution \(q_0\) with heavy tails, reinforcing this idea.  


\subsubsection{A general result for adaptive LISO}

We first provide a convergence rate that is valid for any update of \(\theta_n\), provided that \(\theta_n\) belongs to a set \(\Theta\) satisfying appropriate conditions. For this reason, the procedure for updating \(\theta_n\), and consequently \(q_n\), is left unspecified here. What matters for now is only the family \(\mathcal{Q}\). Let us consider the following assumption on \(\mathcal{Q}\).

\begin{assumption}\label{hyp:Q}
Elements in the family \(\mathcal{Q}\) have sufficiently heavy tails compared to \(\exp(-\alpha f)\); that is, there exists a lower envelope \(g_0 > 0\) and a constant \(c_0 > 0\) such that, for all \(q \in \mathcal{Q}\), \(g_0 \le q\), and for every $y\in\bR^d$
\[
\exp(-(f(y) - f(x^\ast))) \le c_0 \, g_0(y) \,.
\]

Moreover, \(\varphi\) satisfies the following integrability conditions:
    \begin{align*}
        &\int   \exp(-  (f(y) - f(x^* ) )) \|\varphi (y ) \|^2 \,\dr y\\
        &\sup_{q\in \mathcal Q}  \int \|\varphi(y)\|^4 q(y) \,\dr y  < \infty.
    \end{align*}
Finally, the functions $\| \varphi \| ^2/g_0 $, $\| \varphi \| ^2/g_0$, $1/g_0$ are twice-continuously differentiable and $C_{f, \|\bar \varphi \| ^2/g_0 } $, $C_{f, \|\bar \varphi \| ^2/g_0 }$, $C_{f,1/g_0}$ are all positive constants.
\end{assumption}


We have the following upper bound on the variance of adaptive LISO.

\begin{theorem}[general bound for $ \pi_{n,\alpha} ^{(A)} (\varphi) $]\label{th:ALISO}
Let Assumption  \ref{hyp:f} and \ref{hyp:Q} be fulfilled. There exist a constant \( c>0\) and a constant \(C>0\), depending on \(f,g_0,\varphi\), a constant \(\tilde \alpha>1\), depending on \(f,\varphi\), such that for every \(\alpha>\tilde \alpha\) and every \(n\) satisfying
\[
n \geq c {\alpha^{d/2}\log\!\bigl(n{\alpha^{-d/2+1}}(\sqrt n +\alpha^{d/4})\alpha\bigr)} ,
\]
we have
\begin{align*}
&\mathbb E [ \| \pi_{n,\alpha} ^{(A)} (\varphi) - \pi_\alpha (\varphi)  \|_2^2 ] 
\leq 9 \,n^{-1} m_\alpha\,2^{-d/2}
\frac{C_{ f, \|\bar\varphi\|^2/g_0}}{2\alpha}
\left(1+\frac{C}{\alpha}\right).     
\end{align*}

\end{theorem}
Since this result is obtained without any information on the choice of \(q_n\), it is not surprising that it provides a bound that doesn't improve Theorem~\ref{th:1}. In fact, Theorem~\ref{th:ALISO} implies Theorem~\ref{th:1} by choosing a family $\mathcal Q=\{q_0\}$.

The proof is given in the Appendix and is similar to that of Theorem~\ref{th:1}, which considers the fixed policy. A consequence of this result is that the mean squared error of Adaptive LISO in estimating \(\varphi(x^\ast)\) is bounded as follows. This bound shows that any \(\varphi(x^\ast)\) can be estimated by the adaptive LISO. The proof follows that of Corollary~\ref{cor:LISO}, which is provided immediately above its statement.

\begin{coro}\label{cor:ALISO} 
Under the assumptions of Theorem~\ref{th:ALISO}, with Assumption~\ref{hyp:varphi} additionally required, for every $\alpha,n$ satisfying the conditions of Theorem~\ref{th:ALISO}, we obtain
 \begin{equation*}
    \mathbb E [ \| \pi_{n,\alpha}^{(A)} (\varphi) - \varphi(x^*)  \|^2 ] 
\leq  9 n ^{-1}   C_{ f, \|\tilde {\varphi}\| ^2/ g_0}     \alpha^{-1}   \left( 1+\frac{C}{\alpha}\right) + C_{f,\varphi} \alpha ^{-2}  \,.
        \end{equation*}

\end{coro}


Now we consider applying the bound previously obtained to specific families of policies \(\mathcal{Q}\). Our result is valid for two choices:  mixture proposals and bounded policies.

\paragraph{Policies with mixture step}


When $\Theta$ is not compact, the mixture step is necessary for our theory to hold. In this case, as $\theta$ may drift away from $x^\ast$, it is difficult to guarantee that $\exp(-(f - f(x^\ast)))$ is dominated by $q_\theta$. To still maintain control over the variance, we leverage the mixture component, noting that, for all $n \ge 0$,
\(
q_n \ge \lambda q_0
\).
This choice allows us to verify Assumption~\ref{hyp:Q} with \( g_0 = \lambda q_0 \). However, the bound in Corollary~\ref{cor:ALISO} deteriorates as \( \lambda \) becomes too small, which precludes the use of an adaptively decreasing sequence for \( \lambda \). Addressing this limitation is left for future work. An alternative way to circumvent this issue is to restrict attention to bounded policies, as discussed in the next section.


\paragraph{Bounded policies}

A policy is called bounded when \(\Theta\) is a compact set. Each element of \(\Theta\) is used to define the policy \(q_n\). In this case, the update for \(\theta_n\) may be performed by adding a projection step onto the compact set:
\[
\theta_n = \mathrm{Proj}_{\Theta} \bigl(\pi_{n,\alpha}(\mathrm{id})\bigr).
\]
For such policy sampling, the use of the mixture step parameter is not necessary, as indicated by the next lemma.


\begin{lemma}\label{lem:boundede_policy}
    Assume that $ q _\theta (x) = \exp( - \psi_\theta (x )   ) $ with $\psi _\theta:\mathbb R^d \to \mathbb R $ such that $  f(x) - f(x^\ast )  \geq  \psi_{\theta_0} (x )  $, for some $\theta_0 \in \Theta$. Assume also that $\theta \mapsto \psi_\theta (x) $ is $L$-Lipschitz on the compact set $\Theta$, uniformly in $x$. Then, for all $\theta\in\Theta$, we have
    $$ {\exp( -  (f(x) - f(x^*) )) }  \leq c_{\Theta}q_\theta  ,$$ with  $c_{\Theta} = {\sup_{\theta\in\Theta}}\exp (L  \| \theta_0   - \theta \| ) $.
\end{lemma}

One relevant consequence for our problem is that the quantity 
\(\exp(- (f - f(x^\ast))) / q_{\theta}\) 
is bounded, so the tail condition in Assumption~\ref{hyp:Q} is automatically satisfied. One then only needs to require an integrability condition on \(\exp(- f) \, \|\varphi(x)\|\), which is clearly satisfied when, for instance, \(f\) grows sufficiently fast to infinity.


\section{Numerical experiments}
In this section, we compare our methods to some standard sampling-based algorithms on global optimization benchmarks. Additional experiments are in presented Section~\ref{sec:AdditionalPlot} of the appendix.

\begin{figure*}[htbp]
    \centering
    \begin{minipage}{0.3\linewidth}
        \centering
        \includegraphics[width=\linewidth]{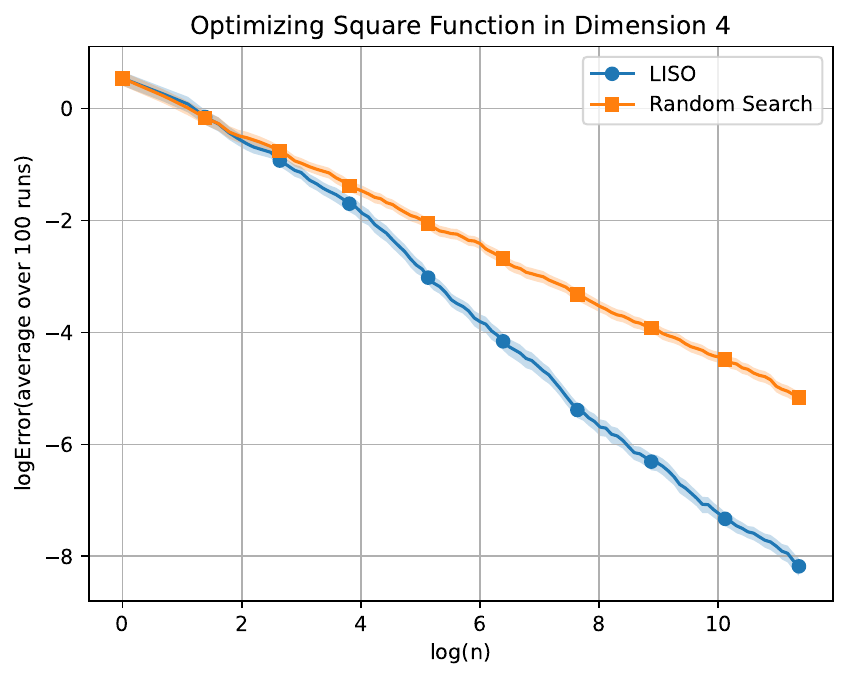}
    \end{minipage}
    \begin{minipage}{0.3\linewidth}
        \centering
        \includegraphics[width=\linewidth]{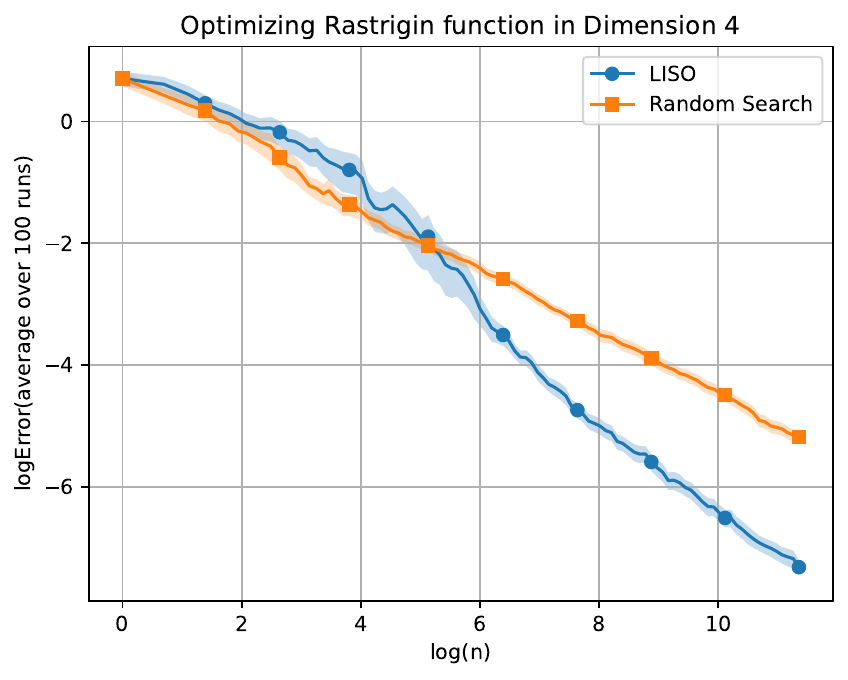}

    \end{minipage}
    \begin{minipage}{0.3\linewidth}
        \centering
        \includegraphics[width=\linewidth]{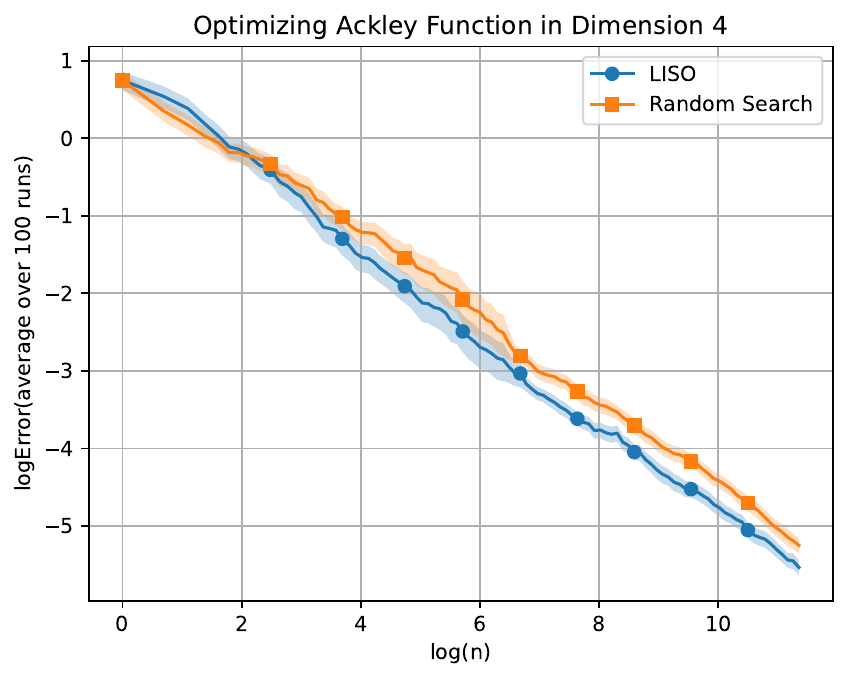}
    \end{minipage}
    \begin{minipage}{0.3\linewidth}
        \centering
        \includegraphics[width=\linewidth]{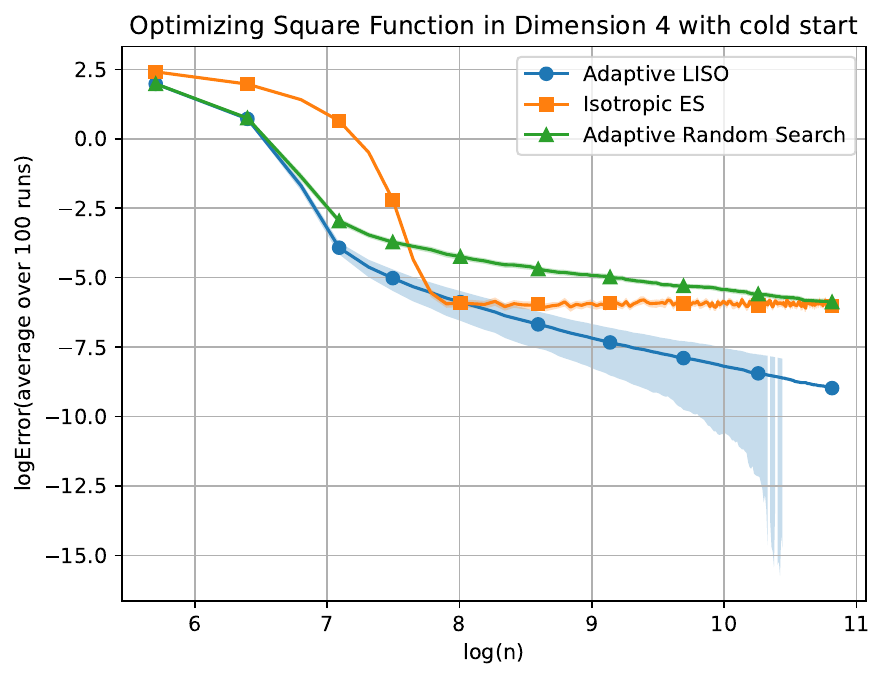}
    \end{minipage}
    \begin{minipage}{0.3\linewidth}
        \centering
        \includegraphics[width=\linewidth]{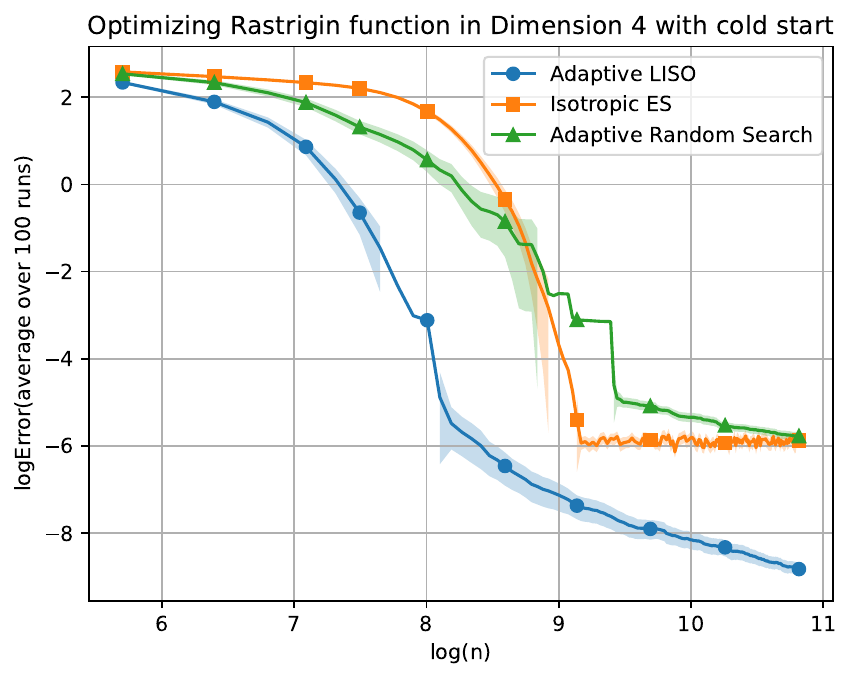}
    \end{minipage}
    \begin{minipage}{0.3\linewidth}
        \centering
        \includegraphics[width=\linewidth]{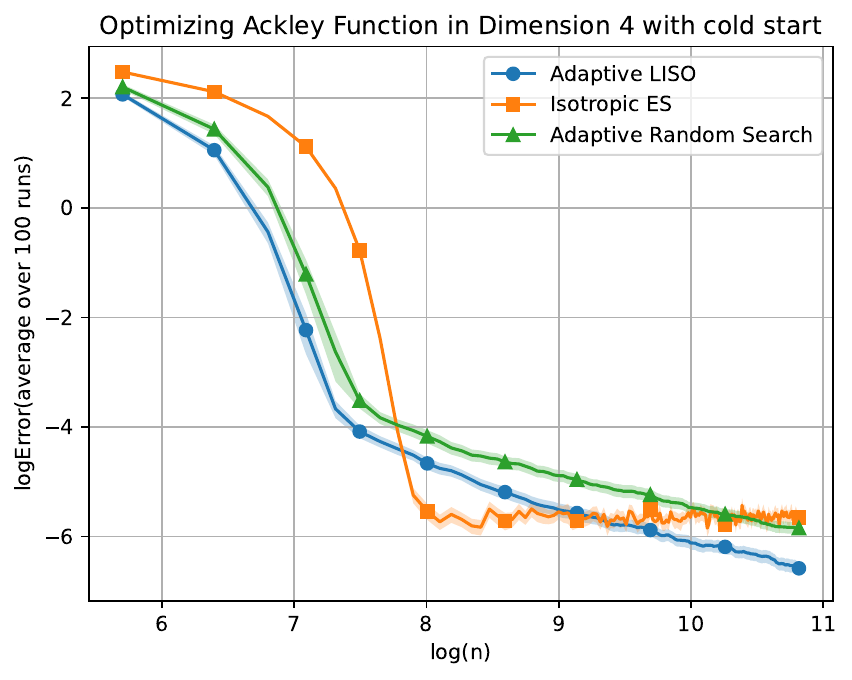}
    \end{minipage}
\caption{The plots on the top refer to static non-adaptive methods, where the initial distribution is close to $x^\ast$ and the toy functions $f_1, f_2, f_3$ are tested in order. The plots on the bottom refer to adaptive methods, where the initial distribution $q_0$ is far from $x^\ast$. The dimension used is $d=4$, and we refer to the other paragraphs in this section for additional information on the experimental parameters. Moreover, Section~\ref{sec:AdditionalPlot} provides the plots at a larger scale, with additional choices of the dimension $d$.  }
\label{fig:result}
\end{figure*}

In the context of zero-order optimization, evaluating the function $f$ often dominates the overall computational cost.
Therefore, a natural way to compare algorithms is to assess their performance in terms of the number of evaluations of $f$.
Since our methods are simple and generic, we do not claim that they outperform sophisticated state-of-the-art algorithms. Our goal is to demonstrate that this inexpensive averaging schemes compares favorably to the strategy taking the global best point, or other weighting techniques, in a wide range of situations.
Since state-of-the-art algorithms, such as CMA-ES, often use some form of averaging as an intermediate step, we wish to demonstrate that our scheme could be used in replacement.

\paragraph{Competitors.} We list the algorithms that we compare. First, consider the two following non-adaptive schemes.
\begin{itemize}
    \item \textbf{LISO}: we consider Algorithm \ref{algo:IS} with a Gaussian sampling distribution $q_0 = \mathcal{N}(\mu, \sigma^2 I_d)$ and the parameter $\alpha \propto n^{\frac{2}{d+2}}$ suggested by theory,
    \item \textbf{Plain random search:} as in Equation \eqref{eq:rs_intro} we sample $n$ point with the same distribution as LISO, and select the best-performing one.
\end{itemize}
Regarding methods with adaptive sampling policies, we compare the three following schemes.
\begin{itemize}
    \item \textbf{Adaptive LISO:} we implement Algorithm \ref{algo:AIS} with a Gaussian sampling policy,
    \item \textbf{Adaptive random search:} at each stage, we sample a new point from a Gaussian distribution centered at the best sample from the previous iterations,
    \item \textbf{Isotropic ES:} this is a simplified variant of the CMA-ES method, in which the covariance matrix is kept fixed throughout the iterations.
We refer to this variant as the \emph{isotropic evolutionary strategy} and the algorithm is described in \cite[Appendix A]{hansen2006cma}. It uses rank-based weighting to select the mean of the next sampling distribution (see details below).
We choose this variant in order to compare the two averaging strategies, all things being equal.
\end{itemize}

Consequently, if Algorithm~\ref{algo:AIS} outperforms rank-based averaging, this suggests that integrating our scheme in CMA-ES may result in improved performances.

\paragraph{Implementation details.}

We run both LISO (Algorithm~\ref{algo:IS}) and plain random search \eqref{eq:rs_intro} with a distribution $q_0=\mathcal N(x^\ast+ d^{-1/2}1_d, d^{-1} I_d)$. We choose the temperature parameter as $\alpha_n =\alpha_0n^{\frac 2{d+2}}$, which is the theoretically-optimal choice (Section \ref{ss:informal}), where the value of $\alpha_0$ will be specified for every examples considered.

For adaptive LISO, adaptive random search, and the isotropic evolutionary strategy, we consider a mini-batch variant.
This choice allows parallelization over the batch size and therefore reduces the computational cost, at the expense of a lower convergence rate. 
We choose a mini-batch size $B=300$.
At iteration $k+1$, the mini-batch version of adaptive LISO samples $B$ i.i.d.\ particles $(X^{k+1,(j)})_{j\in[B]}$ according to
\(
q_k = \mathcal{N}\!\left(\mu_k^{\mathrm{(A)}}, \tfrac{1}{d} I_d\right)
\) (note that the mixture weight $\lambda$ is set to $0$)
The mean is updated as
\(
\mu_k^{\mathrm{(A)}} =
\frac{\sum_{i\in[k]}\sum_{j\in[B]} w_k^{i,(j)} X^{i,(j)}}
{\sum_{i\in[k]}\sum_{j\in[B]} w_k^{i,(j)}},
\)
where the weights are defined by
\(
w_k^{i,(j)} =
\frac{\exp\!\left(-\alpha_{kB} f(X^{i,(j)})\right)}{q_{k-1}(X^{i,(j)})}.
\)
At the end of the algorithm, we return $\mu_n^{\mathrm{(A)}}$.
The mini-batch version of adaptive random search is defined analogously to adaptive LISO, except that $\mu_k^{\mathrm{(A)}}$ is replaced by
\(
\mu_k^{\mathrm{RS}}
:= \arg\min \bigl\{ f(x) \,:\, x \in \{X^{1,(1)}, \dots, X^{k,(B)}\} \bigr\}.
\)
Finally, the mini-batch isotropic evolutionary strategy replaces $\mu_k^{\mathrm{(A)}}$ with
\(
\mu_k^{\mathrm{IES}} := \sum_{i=1}^{\lambda} a^i \tilde X^{k,(i)},
\)
where $(\tilde X^{k,(i)})_{i\in[B]}$ denotes the sequence of sampled particles sorted in increasing order of objective value,
\(
f(\tilde X^{k,(1)}) \le \dots \le f(\tilde X^{k,(B)}),
\)
$\lambda = \lfloor B/2 \rfloor$, and the recombination weights are given by
\(
a^i = \log\!\left(\tfrac{B+1}{2}\right) - \log(i).
\)
This procedure follows the one described in \cite{hansen2006cma}, Table~1, where the covariance update is ignored.
Finally all the algorithm are initialized, with with a distribution $q_0 =\mathcal N(x^\ast + \frac 4{\sqrt{d}}1_d, \frac 1dI_d) $.
This distribution is farther from $x^\ast$ than in the static case, as we wish to demonstrate that the adaptive strategy accomplishes its purpose, namely to adapt from a poor initialization.

\paragraph{Experimental setup.}
For benchmarking our method with respect to its competitors, we consider three canonical test functions: 
the Sphere function ($f_1$), the Rastrigin function ($f_2$), and the Ackley function ($f_3$).
The Sphere function
$
f_1(x) = \norm{x}^2,
$
is a simple convex quadratic function with a unique global minimum at the origin, serving as a baseline to evaluate basic convergence speed in 
well-conditioned scenarios. The Rastrigin function,
\[
f_2(x) = 10d + \sum_{i=1}^d \left[4x_i^2 - 10\cos(\pi x_i)\right],
\]
is a non-convex multimodal function with many local minima, challenging the optimizer's global exploration capabilities.  The Ackley function,
\begin{equation*}
    f_3({x}) =
- 20 \exp\!\left(- 0.2 \sqrt{\frac{1}{d}\sum_{i=1}^{d} x_i^2}\right)
- \exp\!\left(\frac{1}{d}\sum_{i=1}^{d} \cos( 2\pi x_i)\right)
+ 20 + \exp(1),
\end{equation*}

is again a non-convex, multimodal function with many local minima, but this time it is not separable and possesses a narrow valley, i.e., it is ill-conditionned near $x^*$. 

For the functions proposed above, there exists a unique minimizer $x^\ast = 0$, which is an important property for guaranteeing the efficiency of our algorithms; see Assumption~\ref{hyp:f}.


For each method, we report the empirical average of $\lVert x_n - x^\ast \rVert^2$ over 100 runs, where $x_n$ denotes the estimator of $x^\ast$ obtained after $n$ evaluations of $f$.
The results are then plotted on a logarithmic scale with a 95\% confidence interval in order to compare the asymptotic performance of the algorithms.
Moreover, in the case of $f_1$, $\alpha_0 = 1$; in the case of $f_2$, $\alpha_0 = d/80$; and in the case of $f_3$, $\alpha_0 = d/4$.
Simulations are performed for dimensions $d \in \{2,4,8,12\}$.
In this section, we present the results for dimension $d = 4$ and refer to Section~\ref{sec:AdditionalPlot} for additional experiments.

\paragraph{Comments on the numerical experiments.}

The results shown in Figure~\ref{fig:result} show that our method enjoys good performance in both static and adaptive cases.  

It is unsurprising that our algorithm achieves its best performance on the square function $f_1$. Indeed, in this case, the Laplace principle is exact, allowing us to obtain the best theoretical bound (see Theorem~\ref{th:1}). Moreover, experiments on the square function highlight the asymptotic rate predicted by the informal Theorem~\ref{th:informal}, $\mathcal{O}(n^{-4/(d+2)})$, which outperforms the rate of random search, $\mathcal{O}(n^{-2/d})$.  

It is also expected that, compared to the Ackley function $f_3$, the Rastrigin function $f_2$ yields better performance for our algorithms relative to other methods, as it behaves more closely like a quadratic function around $x^\ast$.

The experiments on the adaptive algorithms reveal three phases on adaptive LISO~\ref{algo:AIS}. The first phase is where the distance between $x_n$ and $x^\ast$ decreases slowly while the algorithm searches for the mode around $x^\ast$. In the second phase, the mode is found and the algorithm approaches $x^\ast$ quickly. In the final phase, $x^\ast$ has been located, and the algorithm recovers the asymptotic rates of the static case.  

Finally, we note that the mean update strategy in CMA-ES (isotropic ES) is faster during the second phase, where the mode is found. However, the algorithm eventually ceases to progress. This is because the output of isotropic ES is generated using a fixed number of particles, whereas the outputs of adaptive LISO and adaptive random search are based on an evolving number of particles across iterations.

\section{Discussion}

In this work, we presented a simple averaging strategy that can be easily integrated in sampling-based global optimization methods. We showed that this approach enjoys favorable theoretical guarantees in terms of asymptotic dependence on $n$ and $d$, and demonstrates good practical performance on benchmark problems.

An important limitation is that we assume that the global minimizer $x^*$ is unique, and that the objective function is locally strongly convex in a neighborhood of $x^*$. Without this assumption, both the theory and practice may fail, as the method crucially relies on convex combinations. An important question for future work is to handle functions which are not locally strongly convex but still enjoy a weaker form of strong convexity when restricted to some directions, such as Polyak-Łojasiewicz functions \cite{PL2016}.

Another promising direction is to integrate this scheme into more complex adaptive methods, which also adjust the covariance matrix of the sampling distribution (in the spirit of CMA-ES) in order to better handle ill-conditionned landscapes, and to derive sharp convergence guarantees.

\bibliographystyle{alpha}
\bibliography{bib}

\newpage
\appendix

\section{Proof of Theorem~\ref{th:1}}
\label{sec:proofTh1}

Without loss of generality, we can assume $f(x^\ast) = 0$. The proof is in two steps. First we prove Proposition \ref{prop:unnormalized}. Then we consider proving the result of interest.

\subsection{Proof of Proposition~\ref{prop:unnormalized}}
\label{sec:ProofProp}
 Let us begin by applying Lemma~\ref{lemma:laplace_approx} to obtain two auxiliary bounds. 
We note that by Assumption~\ref{hyp:f}, Lemma~\ref{lemma:laplace_approx} holds with $\varphi=1$.
There exist $C$ depending on $f$ only such that for all $\alpha > 1$,
\[
\abs{Z_{\alpha} - m_\alpha^{-1}} \le m_\alpha^{-1}\frac {C}\alpha \quad \text{and} \quad \abs{Z_{2\alpha} - 2^{-d/2}m_\alpha^{-1}} \le  2^{-d/2}m_\alpha^{-1}\frac {C}{2\alpha}\,.
\]
and more particularly 
\[
 m_\alpha^{-1} \left(1 - \frac {C}\alpha \right) \le Z_{\alpha}  \quad \text{and} \quad Z_{2\alpha} \le  2^{-d/2}m_\alpha^{-1} \left(1+  \frac {C}{2\alpha}\right)\,.
\]
We take $\tilde \alpha_1$ large enough such that for every $\alpha>\tilde \alpha_1$, we have $C/\alpha\le 1 / 2  $.  When $x\in[0,  1 / 2]$, we have 
$$(1-x)^{-2}(1+x/2) \leq (1+6 x) (1+x/2) \leq (1+8x) \,.$$
Then, we obtain: 
\begin{align}\label{eq1}
    \frac{Z_{2\alpha}}{Z_{\alpha}^2}  
    \le  2^{-d/2}  m_{\alpha} \left( 1+ 8 \frac C \alpha \right).
\end{align} 
By Assumption~\ref{hyp:q0}, the assumptions of Lemma~\ref{lemma:laplace_approx} hold with $\norm{\varphi}^2/q_0$.
Then, from Lemma \ref{lemma:laplace_approx}, there exists $\tilde\alpha_2$ depending on $f,\varphi$ such that it holds, for any $\alpha \geq \tilde \alpha_2$,
\begin{align}\label{eq2}
    \int  {\frac{ \norm{\varphi(x)} ^2}{q_0(x)}}\pi_{2\alpha}(\dr x)  \le \frac{\norm{\varphi(x^\ast)}^2}{q_0(x^\ast)} + \frac{C_{f,\norm{\varphi} ^2 /q_0 } }{2 \alpha}\,.
\end{align}
Let us work under the condition that
\(
\alpha \ge  \tilde{\alpha}
:= \tilde{\alpha}_1 \vee \tilde{\alpha}_2 \vee 1
\).
Now observe that, by definition of $\tilde{\pi}_{n,\alpha}$,

\[
\tilde \pi_{n,\alpha} = \frac 1n\sum_{i\in[n]}  \bar w^i\delta_{X^i}\,,
\]
with $\bar w^i = {\pi_\alpha(X^i)} / {q_0(X^i)}$. Then, using that the random variables $(X^i)_{i\in[n]}$ are i.i.d. and $\bE (\bar w^i\varphi(X^i) )=\pi_\alpha(\varphi) $:
\begin{align*}
    \bE\norm{\tilde \pi_{n,\alpha}(\varphi) - \pi_{\alpha}(\varphi)}^2 &= \bE \norm{\frac 1n\sum_{i\in [n]} (\bar w^i\varphi(X^i) - \pi_\alpha(\varphi))}^2\\
    &=\frac 1n\bE \norm{\bar w^1\varphi(X^1) -\pi_\alpha(\varphi)}^2\\
    &\le \frac 1 n   \bE\norm{\bar w^1\varphi(X^1)}^2  \,.
\end{align*}
We complete the proof of the first statement by showing that
\begin{align}\label{eq:varbarw}
    \bE\norm{\bar w^1\varphi(X^1)}^2 \leq  2^{-d/2}  m_{\alpha}  \left( \frac{\norm{\varphi(x^\ast)}^2}{q_0(x^\ast)} + \frac{C_{f,\varphi ^2 /q_0 } }{2 \alpha} \right)( 1+ \frac {C_1}\alpha  ) \,,
\end{align}
for $C_1$ that depends on $f$ and $\varphi$ only. 
Since $\pi_\alpha^2 = \pi_{2\alpha} {Z_{2\alpha}} /{Z_\alpha^2}$, it follows that 
$$  \bE\norm{\bar w^1\varphi(X^1)}^2 
 =  \int {\frac{ \norm{\varphi(x)} ^2}{q_0(x)}}\pi_{\alpha}^2(x)\dr x = \frac{Z_{2\alpha}}{Z_{\alpha}^2}\int  {\frac{ \norm{\varphi(x)} ^2}{q_0(x)}}\pi_{2\alpha}(\dr x)\,, $$
and by combining \eqref{eq1} and \eqref{eq2}, we obtain
\begin{align*}
 \bE\norm{\bar w^1\varphi(X^1)}^2 
&\le\frac 1n \frac{Z_{2\alpha}}{Z_{\alpha}^2}\int  {\frac{ \norm{\varphi(x)} ^2}{q_0(x)}}\pi_{2\alpha}(\dr x) \\
&\le \frac 1n  2^{-d/2}  m_{\alpha} ( 1+ 8 C\alpha^{-1} ) \int  {\frac{ \norm{\varphi(x)} ^2}{q_0(x)}}\pi_{2\alpha}(\dr x) \\
&\le \frac 1n  2^{-d/2}  m_{\alpha} ( 1+ 8 C\alpha^{-1}  )  \left( \frac{\norm{\varphi(x^\ast)}^2}{q_0(x^\ast)} + \frac{C_{f,\norm{\varphi} ^2 /q_0 } }{2 \alpha} \right)\,,
\end{align*}
We obtain \eqref{eq:varbarw}, taking $C_1=8C$ and the proof is finished.

\subsection{End of the proof}\label{subsec:Proofth}

We introduce a sequence of positive constants $C_1, \dots, C_6$ and $\tilde{\alpha}_1, \tilde{\alpha}_2$, which are not made explicit. However, we tried to keep track of the dependence between each of them.

We define $\bar \varphi = \varphi -\varphi(x^\ast)$ and we remark that
    $$\pi_{n,\alpha}(\varphi) - \pi_{\alpha}(\varphi) = \pi_{n,\alpha}(\bar \varphi) -\pi_{\alpha}(\bar \varphi)  = \frac{\tilde \pi_{n,\alpha}(\varphi) -\tilde \pi_{n,\alpha}(1)\pi_{\alpha}(\bar\varphi)}{\tilde \pi_{n,\alpha}(1)}\,,  $$
   and defining the event $A:= \{\tilde \pi_{n,\alpha}(1) \geq 1/2 \}$, we obtain
    \begin{align*}
        \norm{\pi_{n,\alpha}(\varphi) - \pi_{\alpha}(\varphi)}^2&= \norm{\pi_{n,\alpha}(\varphi) - \pi_{\alpha}(\varphi)}^2\un_{ A} +    \norm{\pi_{n,\alpha}(\varphi) - \pi_{\alpha}(\varphi)}^2 \un_{ A^c}\\
        &\leq  4\norm{\tilde \pi_{n,\alpha}(\bar \varphi)  - \pi_{\alpha}(\bar \varphi) \tilde\pi_{\alpha,n}(1) }^2 \un_{ A} + \norm{\pi_{n,\alpha}(\varphi) - \pi_{\alpha}(\varphi)}^2 \un_{ A^c}.
    \end{align*}  
We take the expectation and use the Cauchy-Schwartz inequality, to get
\begin{equation}\label{eq:term-to-bound-th3}
    \mathbb E \norm{\pi_{n,\alpha}(\varphi) - \pi_{\alpha}(\varphi)}^2\leq 4 \mathbb E  \norm{\tilde \pi_{n,\alpha}(\bar \varphi)  - \pi_{\alpha}(\bar \varphi) \tilde\pi_{\alpha,n}(1) }^2 + \sqrt{\bE \norm{\pi_{n,\alpha}(\varphi) - \pi_{\alpha}(\varphi)}^4} \sqrt{\bP(A^c)}.
\end{equation}
We then have three terms to study. For the first term, we get
\[
    \mathbb E  \norm{\tilde \pi_{n,\alpha}(\bar \varphi ) - \pi_{\alpha}(\bar \varphi) \tilde\pi_{n,\alpha}(1) }^2
        \le 2\mathbb E  \norm{\tilde \pi_{n,\alpha}(\bar \varphi ) - \pi_{\alpha}(\bar \varphi)}^2 + 2\norm{\pi_{\alpha}(\bar \varphi)}^2\mathbb E  \norm{\tilde\pi_{n,\alpha}(1)-\pi_{\alpha}(1)}^2
\]
The two terms above are treated by relying on Proposition~\ref{prop:unnormalized}, with the test functions $\bar{\varphi}$ and $1$, respectively, as they both satisfy the assumptions of Proposition~\ref{prop:unnormalized}. Since $\bar{\varphi}(x^*) = 0$, the function $\bar{\varphi}$ enjoys an improved convergence rate compared to the standard case $\varphi(x^*) \neq 0$. More concretely, there exist constants $\tilde{\alpha}_1 > 1$, depending on $f$, $\varphi$, and $q_0$, and $C_1$, depending on $f$ and $q_0$, such that for all $\alpha \ge \tilde{\alpha}_1 $,
\[
\mathbb{E}\!\left[\bigl\|\tilde \pi_{n,\alpha}(\bar \varphi)-\pi_{\alpha}(\bar \varphi)  \bigr\|^2\right]
\le n^{-1} 2^{-d/2}  m_{\alpha}\frac{C_{f,\norm{\bar \varphi} ^2 /q_0 } }{2 \alpha}  \left( 1+   \frac {C_1} \alpha   \right)  \,,
\]
and 
\[
\bE \left[\norm{\tilde \pi_{n,\alpha}(1) - \pi_{\alpha}(1)}^2\right]
\le n^{-1} 2^{-d/2}  m_{\alpha}\left( \frac{1}{q_0(x^\ast)} + \frac{C_{f,1 /q_0 } }{2 \alpha} \right)  \left( 1+   \frac {C_1} \alpha   \right) 
\,.
\]
Moreover by Lemma~\ref{lemma:laplace_approx}, we have
\[
\norm{\pi_{\alpha}(\bar \varphi)}^2 \le  \frac {C_{f,\bar \varphi}^2}{\alpha^2}\,.
\]
It follows that there exists a constant $C_2>0$ depending on $f,\varphi,q_0$ such that
\begin{equation}
\label{eq:first-term-th3}
\begin{split}
     \mathbb E  \norm{\tilde \pi_{n,\alpha}(\bar \varphi)  - \pi_{\alpha}(\bar \varphi) \tilde\pi_{\alpha,n}(1)  }^2 &\le  2 
     n^{-1} 2^{-d/2}  m_{\alpha}   \left(\frac{C_{f,\bar \varphi ^2 /q_0 } }{2 \alpha} + \frac {C_{f,\bar \varphi}^2}{\alpha^2}\left( \frac{1}{q_0(x^\ast)} + \frac{C_{f,1 /q_0 } }{2 \alpha} \right)\right)
    \left( 1+   \frac {C_1} \alpha   \right) \\
     & \le  2 
     n^{-1} 2^{-d/2}  m_{\alpha}  \left(\frac{C_{f,\bar \varphi ^2 /q_0 } }{ \alpha} \right)\left(1+\frac { {C_2}}\alpha\right)\,.
\end{split}
\end{equation}

We now proceed to bound the second term in  Equation~\eqref{eq:term-to-bound-th3}:
\[
\bE \norm{\pi_{n,\alpha}(\varphi) - \pi_{\alpha}(\varphi)}^4 \le 8\bE\norm{\pi_{n,\alpha}(\varphi)}^4 + 8\norm{\pi_{\alpha}(\varphi)}^4\,.
\]
Then, using that the random variables $(X^i)_{i\in[n]}$ are i.i.d., we obtain:
\[
\begin{split}
    \bE\norm{\pi_{n,\alpha}(\varphi)}^4 &\le \bE\p{\frac {\sum_{i\in[n]}w^i\norm{\varphi(X^i)}^4}{\sum_{i\in[n]} w^i}}\\
    &\le \bE \p{\max_{i\in[n]} \norm{\varphi(X^i)}^4}\\
    &\le n\bE \norm{\varphi(X^1)}^4\,,
\end{split}
\]
and 
$$\norm{\pi_{\alpha}(\varphi)}^4\leq \pi_{\alpha}(\norm{ \varphi}^4)  =  Z_\alpha ^{-1}  \int \exp(-\alpha (f(x)-f(x^\ast) ) ) \| \varphi(x) \|^4\,\dr x  $$
we obtain with 
\[
\norm{\pi_{\alpha}(\varphi)}^4 \le Z_\alpha ^{-1}   \int \exp(- (f(x)-f(x^\ast) ) ) \| \varphi(x) \|^4\,\dr x\,,
\]
then we obtain with Assumption~\ref{hyp:varphi}, the existence of a constant $C_3>0$ depending on $\varphi,f$ such that
\[
\norm{\pi_{\alpha}(\varphi)}^4 \le C_3 m_\alpha\,.
\]
Combining the two computations, we obtain:
\begin{equation}
    \label{eq:second-term-th3}
   \sqrt{ \bE \norm{\pi_{n,\alpha}(\varphi) - \pi_{\alpha}(\varphi)}^4 }\le \sqrt{8(n\bE \norm{\varphi(X^1)}^4 + C_3m_\alpha)}\,.
\end{equation}

We bound the third term of Equation~\eqref{eq:term-to-bound-th3} by Bernstein inequality. Indeed, we have:
\[
\bP(A^c) = \bP \left( \pi_{n,\alpha}(1) \le \frac 12 \right) = \bP\left( \frac 1n \sum \bar w^i \le \frac 12\right) \,,
\]
with the notation $\bar w^i = \pi_{\alpha}(X^i)/q_0(X^i)$.
Now, we state Bernstein inequality in our context.
\begin{lemma}\label{lem:Bernstein}
    Let $(\bar w^1,\ldots, \bar w^n)$ be independent and identically distributed random variables such that $ 0 \leq \bar w^1 \leq U$ a.s.,  $\mathbb E (\bar w^1) = 1$ and  $\mathbb E (\bar w^i)^2 = v$. We have:
$$ \mathbb P \left(  n^{-1} \sum_{i\in[n]}  \bar w^i < 1/2 \right)  \leq \exp \left(- \frac{ n }{ 8v + 4U/3  } \right) .   $$
\end{lemma}
\begin{proof}
Note that by assumption $ - (\bar w^1 - 1)$ is a centered random variable such that $ -1 \leq \bar w^1 - 1 \leq U -1\leq U $ which implies that $|\bar w^1-1| \leq U\vee 1 = U $ (because $U\geq 1$ by the assumptions). We can apply Bernstein inequality to obtain that 
$$ \mathbb P \left(   - \sum_{i=1}^n ( \bar w^i -1) > t \right)  \leq \exp \left(- \frac{t^2/2 }{   n v + Ut/3} \right) . $$
    Taking $t = n/2$ we obtain the statement.
\end{proof}
We then identify $U := \sup_{\omega \in \Omega} \bar w^i(\omega)$ and $v := \esp{(\bar w^i)^2}$.
Note that, with Assumption~\ref{hyp:q0} there exists a constant $c$ depending on $f,q_0$ such that:
\[
{\ex^{-\alpha f}}<c q_0\,.
\]
By Equation~\eqref{eq1}, there exists a constant $C_4>0$ depending on $f$ such that: 
\[
\sup_{\omega \in \Omega} \bar w^i(\omega) \le cZ_\alpha^{-1}\le  c m_{\alpha}\left(1+ \frac {C_4}\alpha \right) =: U \,.
\]
Moreover, using Equation~\eqref{eq:varbarw}, there exists a constant $C_5>0$ depending on $f,q_0$ such that:
\[
\bE\p{\bar w^1 }^2 \le 2^{-d/2}m_\alpha   \p{q_0(x^*) ^{-1}+ \frac {C_5}\alpha} =: v\,.
\]
Then, Lemma~\ref{lem:Bernstein}, yields the existence of a constant $C_6>0$ depending on $f,q_0$:
\begin{equation}
\label{eq:third-termr-th3}
    \bP(A^c) \le  \exp\p{-2n m_\alpha^{-1}  C_6}\,.
\end{equation}
Putting together Equation~\eqref{eq:first-term-th3},~\eqref{eq:second-term-th3}, and~\eqref{eq:third-termr-th3} in Equation~\eqref{eq:term-to-bound-th3}, we obtain for every $\alpha>\tilde\alpha_1\vee\tilde\alpha_2 =: \tilde\alpha$:
\begin{multline*}
      \mathbb E \norm{\pi_{n,\alpha}(\varphi) - \pi_{\alpha}(\varphi)}^2\leq 8n^{-1} m_\alpha 2^{-d/2}
\left(\frac {C^\mathrm{Lap}_{{\norm{\bar\varphi}^2}/{q_0}}}{2\alpha}\right)(1+\frac { {C_2}}\alpha) \\+ \sqrt{8(n\bE \norm{\varphi(X^1)}^4 + C_3m_\alpha)} \exp\p{-n m_\alpha^{-1}  C_6}\,.
\end{multline*}
Then, there exists a constant $\tilde c>0$ and $C>0$ depending on $q_0,f,\varphi,\alpha_0$ such that for every $(n,\alpha)$ such that $nm_\alpha^{-1} \ge \tilde c \log(n m_\alpha^{-1}\alpha(\sqrt n +\sqrt{m_\alpha}) )$, we obtain for every $\alpha>\tilde \alpha $:
\[
 \mathbb E \norm{\pi_{n,\alpha}(\varphi) - \pi_{\alpha}(\varphi)}^2\leq 9n^{-1} m_\alpha 2^{-d/2}
\left(\frac {C_{f,{\norm{\bar\varphi}^2}/{q_0}}}{2\alpha}\right)(1+\frac { {C}}\alpha)\,.
\]

\section{Proof of Theorem \ref{th:ALISO}}

Without loss of generality, we can assume $f(x^\ast) = 0$. As for the proof of Theorem \ref{th:1}, the proof is in two steps. First we prove the next proposition and only then we consider proving the result of interest.

We introduce the adaptive \emph{unnormalized version} of $\pi_{n,\alpha}^{(A)}(\varphi)$:
\begin{align*}
    \tilde \pi_{n,\alpha}^{(A)}(\varphi) = \frac{1}{n Z_\alpha} \sum_{i \in [n]} w^i \, \varphi(X^i),
\end{align*}
The following is an upper bound on the variance of $ \tilde \pi_{n,\alpha}^{(A)}(\varphi)  $.

\begin{prop}[general bound for $\tilde \pi_{n,\alpha}^{(A)} (\varphi) $]
Suppose that Assumption \ref{hyp:f} and \ref{hyp:Q} are fulfilled. We have for any $\alpha >\tilde \alpha_1$,
\begin{align*}
    &\mathbb E [ \|\tilde \pi_{n,\alpha}^{(A)} (\varphi)- \pi_\alpha (\varphi)  \|^2 ] \leq n ^{-1}  2^{-d/2}m_\alpha   \left( \frac{\norm{\varphi(x^\ast)}^2}{g_0(x^\ast)}+\frac {C_{f,\norm{\varphi}^2/g_0}}{2\alpha}\right)\p{1+\frac {C_1}\alpha}\,,
\end{align*}
where $C_1, \tilde \alpha_1 $ depends only on $g_0,f,\varphi$.
\end{prop}

\begin{proof}
Note that
    \[
    \begin{split}
     \tilde \pi_{n,\alpha}^{(A)} (\varphi) - \pi_\alpha (\varphi) =\frac 1 {nZ_\alpha} \sum_{i\in[n]} \xi^i\,,
    \end{split}
    \]
    where $\xi^i :=  w^i \varphi(X^i) - Z_\alpha \pi_\alpha(\varphi) $.     Let $\cF_i$ be the $\sigma$-algebra generated by $(X_k)_{1\le k\le i}$.
    The $(\xi^i)$ are martingale increments satisfying $\espcond{\xi^i}{\cF_{i-1}} = 0$ and
    \[
    \begin{split}
            \espcond{\norm{\xi^i}^2}{\cF_{i-1}} &\le \espcond{\norm{ w^i \varphi(X^i)}^2}{\cF_{k-1}} \\
&\leq Z_{2\alpha} \int \frac{\norm{\varphi (x)} ^2 }{q_{i-1}(x) } \pi_{2\alpha} (dx) \\
&\leq   Z_{2\alpha} \int \frac{\norm{\varphi (x)} ^2 }{    g_{0}(x) } \pi_{2\alpha} (dx)\,. \\
    \end{split}
    \]
Then, using Equation~\eqref{eq1} and \eqref{eq2}, we obtain the existence of constants $C_1,\tilde \alpha_1$ depending on $g_0,f,\varphi$ such that for every $i\in[n]$, it holds for any $\alpha>\tilde \alpha_1$:
\[
 \espcond{\norm{\xi^i}^2}{\cF_{i-1}}\le    Z_{\alpha}^2 2^{-d/2}m_\alpha   \left( \frac{\norm{\varphi(x^\ast)}^2}{g_0(x^\ast)}+\frac {C_{f,\norm{\varphi}^2/g_0}}{2\alpha}\right)\p{1+\frac {C_1}\alpha}\,,
\]
and the result follows.
\end{proof}

\paragraph{End of the proof}  The end of the proof follows in the same way as that of Theorem \ref{th:1} except that instead of the Bernstein we use the Freedman inequality. The changes involved are straightforward. One last difference is in the the treatment of 
\[
\bE \norm{\pi_{n,\alpha}(\varphi) - \pi_{\alpha}(\varphi)}^4 \le 8\bE\norm{\pi_{n,\alpha}(\varphi)}^4 + 8\norm{\pi_{\alpha}(\varphi)}^4\,.
\]
We have
\[
\begin{split}
    \bE\norm{\pi_{n,\alpha}(\varphi)}^4 &\le \bE\p{\frac {\sum_{i\in[n]}w^i\norm{\varphi(X^i)}^4}{\sum_{i\in[n]} w^i}}\\
    &\le \bE \p{\max_{i\in[n]} \norm{\varphi(X^i)}^4}\\
    &\le \sum_{i=1} ^n \int \norm{\varphi(x)}^4 q_{i-1}(x)\dr x\\
    &\le n \sup_{q\in \mathcal Q} \int \norm{\varphi(x)}^4 g_0(x)\dr x\, .
\end{split}
\]
Similarly to the proof of Theorem \ref{th:1},  we then obtain:
\begin{equation*}
   \sqrt{ \bE \norm{\pi_{n,\alpha}(\varphi) - \pi_{\alpha}(\varphi)}^4 }\le \sqrt{8( n \sup_{q\in \mathcal Q} \int \norm{\varphi}^4 q + C_3m_\alpha)}\,.
\end{equation*}

\section{Proof of Lemma \ref{lem:boundede_policy}}
Without loss of generality, we assume $f(x^\ast)=0$.
Then, we write
    \begin{align*}
         {\exp( -  f) }  \leq \exp\left ( -          \psi_{\theta_0}   \right)   \,.
    \end{align*}
    With the definition $q_\theta =\exp(-\psi_\theta) $, the Lipschitz condition yields 
    $$ \frac{\exp( -  f) }{q_\theta}\leq \exp\left (  \psi_\theta  -\psi_{\theta_0} \right)\le\exp(L\norm{\theta-\theta_0}) , $$
and we can conclude.

\section{Additional experiments}
\label{sec:AdditionalPlot}
\begin{figure*}[ht!]
    \centering

    \begin{subfigure}[b]{0.45\textwidth}
        \centering  
        \includegraphics[width=\textwidth]{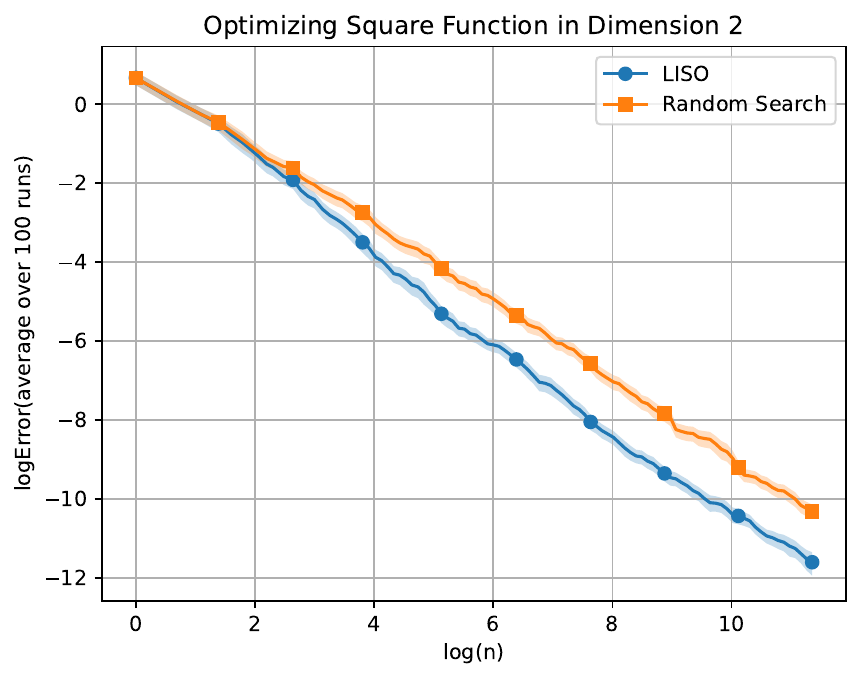}
    \end{subfigure}
    \hfill
    \begin{subfigure}[b]{0.45\textwidth}
        \centering
        \includegraphics[width=\textwidth]{Optimizing_Square_Function_in_Dimension_4.pdf}

    \end{subfigure}

    \vskip\baselineskip

    \begin{subfigure}[b]{0.45\textwidth}
        \centering
        \includegraphics[width=\textwidth]{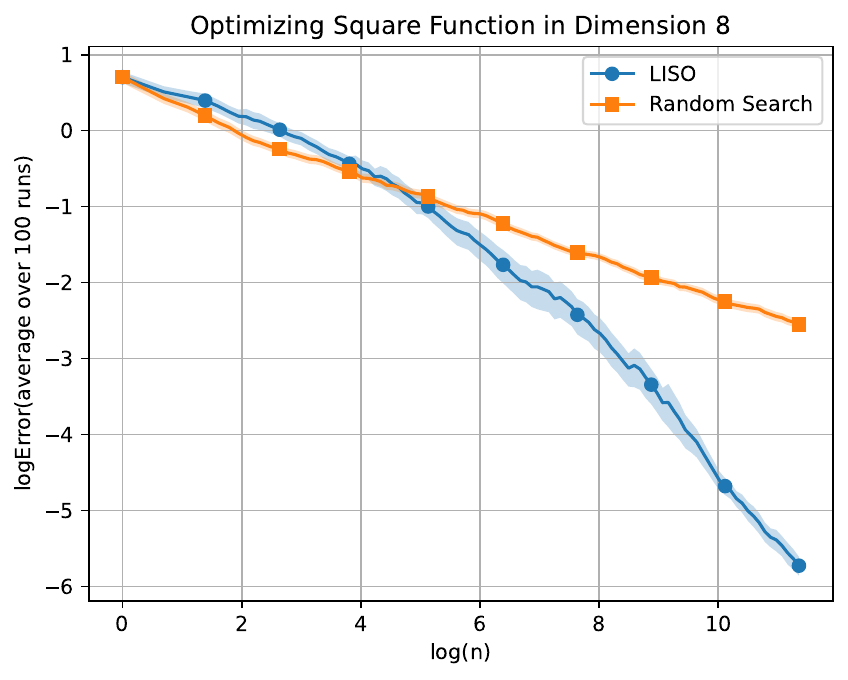}

    \end{subfigure}
    \hfill
    \begin{subfigure}[b]{0.45\textwidth}
        \centering
        \includegraphics[width=\textwidth]{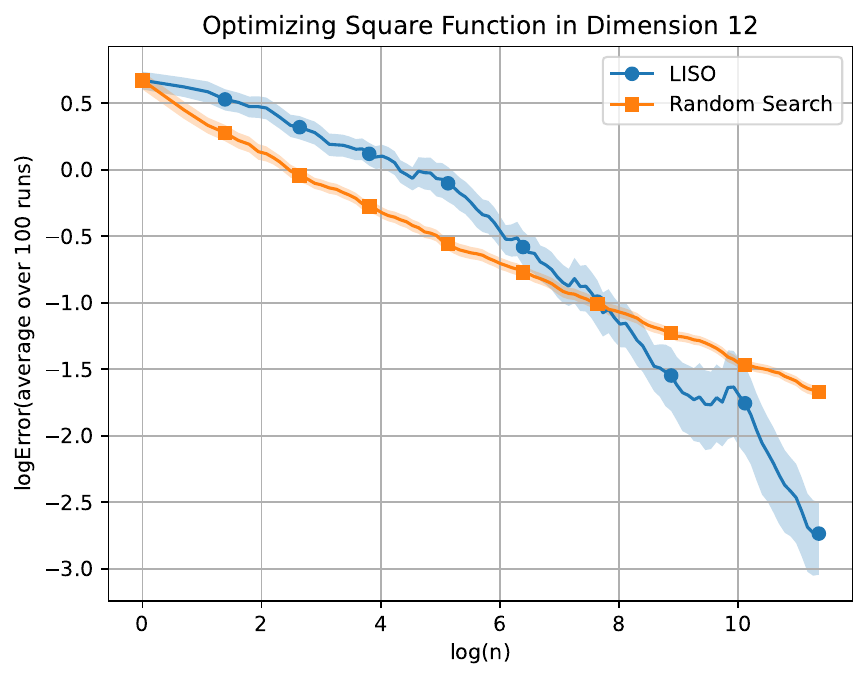}

    \end{subfigure}

    \caption{Function Square ($f_1$) in the static case.}
\end{figure*}

\begin{figure*}[ht!]
    \centering

    \begin{subfigure}[b]{0.45\textwidth}
        \centering  
        \includegraphics[width=\textwidth]{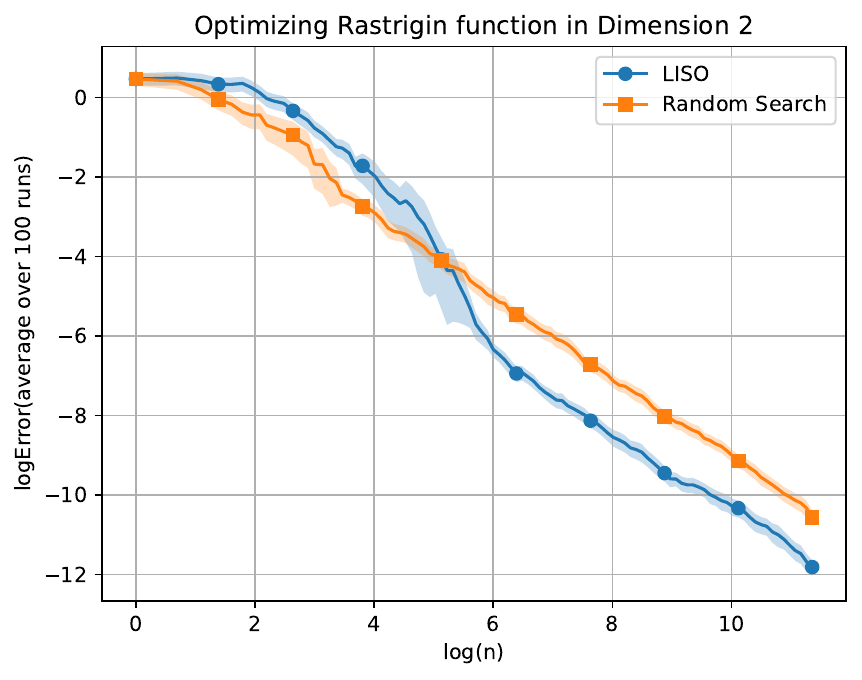}

    \end{subfigure}
    \hfill
    \begin{subfigure}[b]{0.45\textwidth}
        \centering
        \includegraphics[width=\textwidth]{Optimizing_Rastrigin_Function_in_Dimension_4.pdf}
    \end{subfigure}

    \vskip\baselineskip

    \begin{subfigure}[b]{0.45\textwidth}
        \centering
        \includegraphics[width=\textwidth]{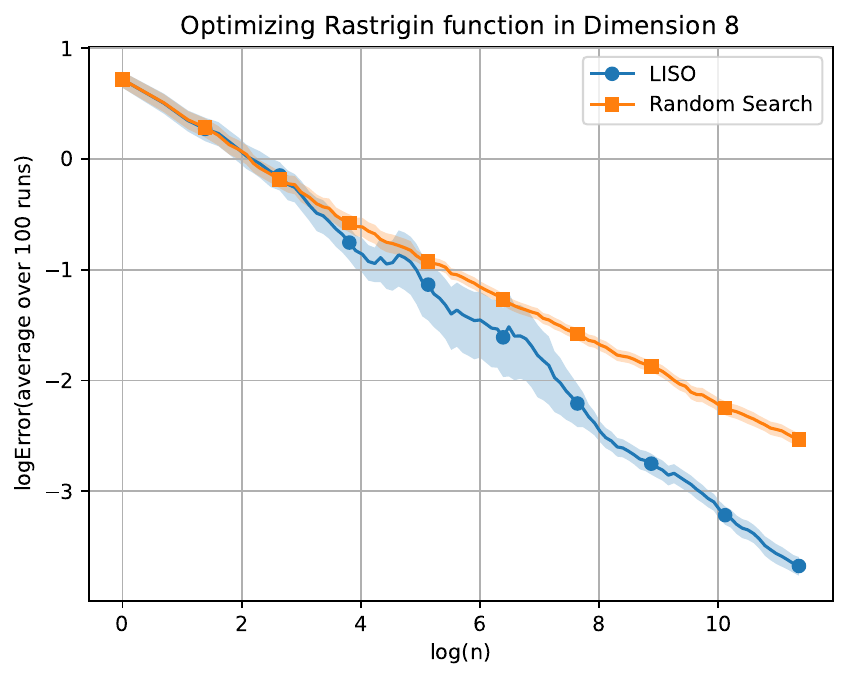}
    \end{subfigure}
    \hfill
    \begin{subfigure}[b]{0.45\textwidth}
        \centering
        \includegraphics[width=\textwidth]{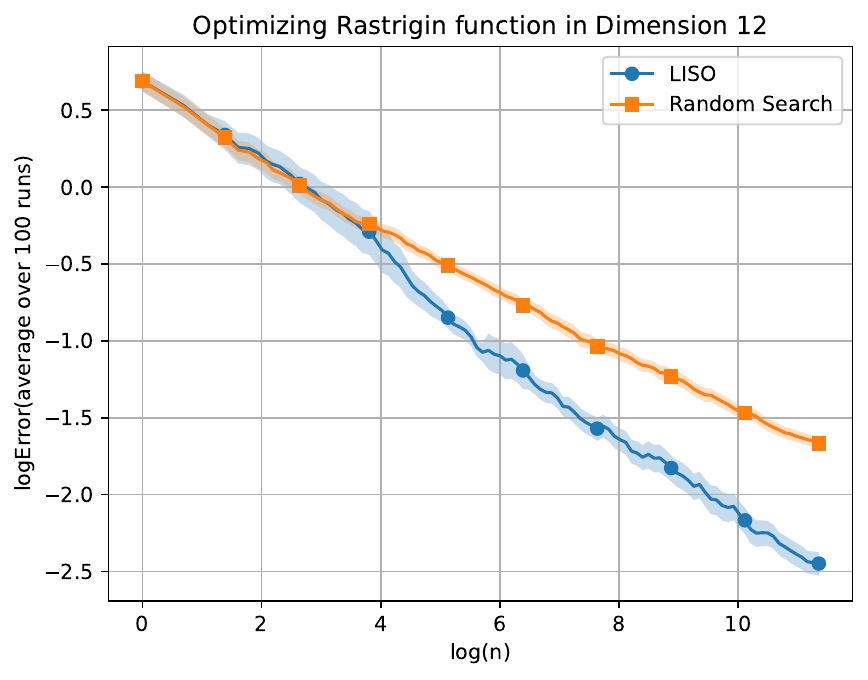}

    \end{subfigure}

    \caption{Function Rastrigin ($f_2$) in the static case. }
\end{figure*}

\begin{figure*}[ht!]
    \centering

    \begin{subfigure}[b]{0.45\textwidth}
        \centering  
        \includegraphics[width=\textwidth]{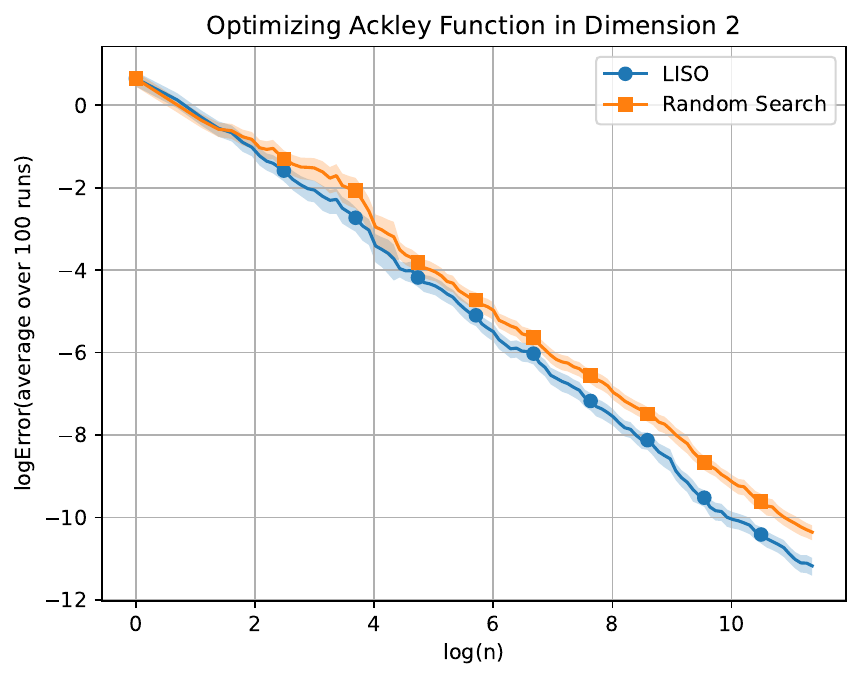}

    \end{subfigure}
    \hfill
    \begin{subfigure}[b]{0.45\textwidth}
        \centering
        \includegraphics[width=\textwidth]{Optimizing_Ackley_Function_in_Dimension_4.pdf}

    \end{subfigure}

    \vskip\baselineskip

    \begin{subfigure}[b]{0.45\textwidth}
        \centering
        \includegraphics[width=\textwidth]{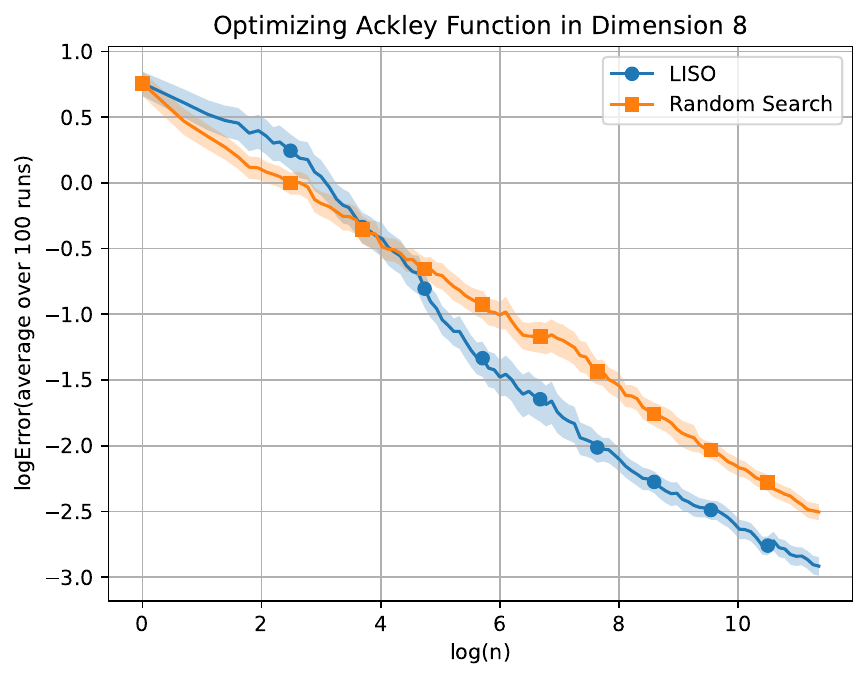}
  
    \end{subfigure}
    \hfill
    \begin{subfigure}[b]{0.45\textwidth}
        \centering
        \includegraphics[width=\textwidth]{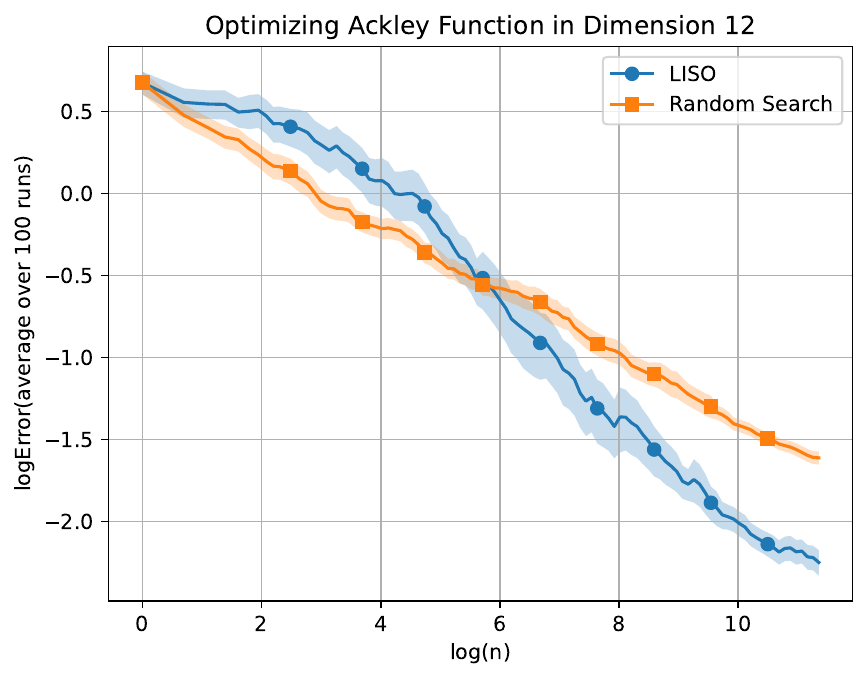}
   
    \end{subfigure}

    \caption{Function Ackley ($f_3$) in the static case.}
\end{figure*}

\begin{figure*}[ht!]
    \centering

    \begin{subfigure}[b]{0.45\textwidth}
        \centering  
        \includegraphics[width=\textwidth]{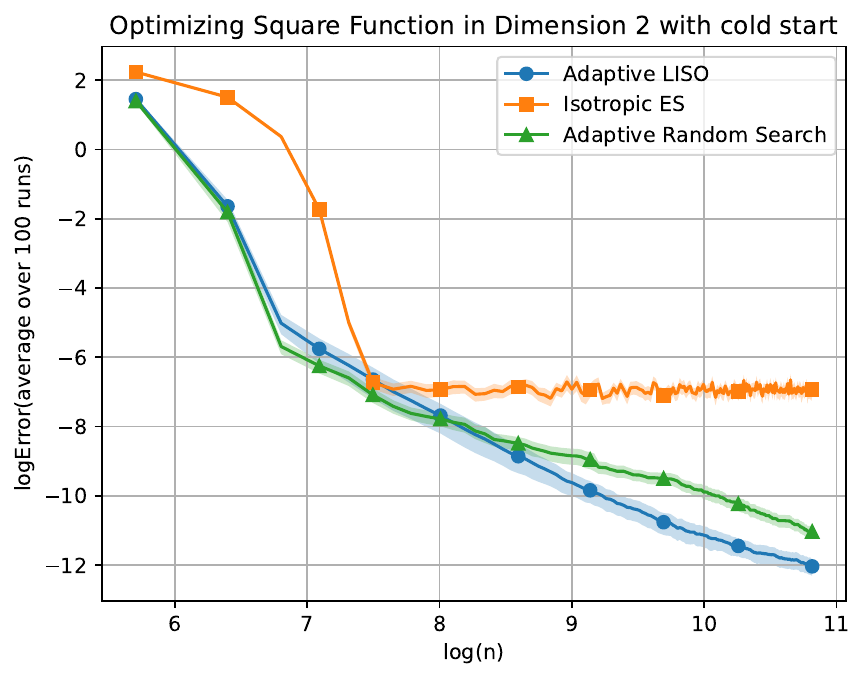}

    \end{subfigure}
    \hfill
    \begin{subfigure}[b]{0.45\textwidth}
        \centering
        \includegraphics[width=\textwidth]{Optimizing_Square_Function_in_Dimension_4_with_cold_start.pdf}

    \end{subfigure}

    \vskip\baselineskip

    \begin{subfigure}[b]{0.45\textwidth}
        \centering
        \includegraphics[width=\textwidth]{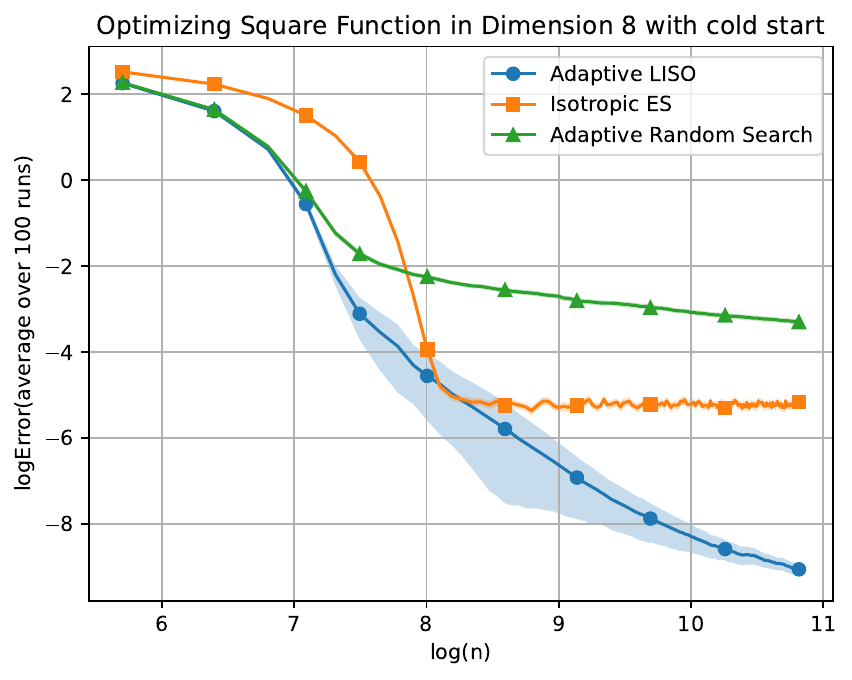}

    \end{subfigure}
    \hfill
    \begin{subfigure}[b]{0.45\textwidth}
        \centering
        \includegraphics[width=\textwidth]{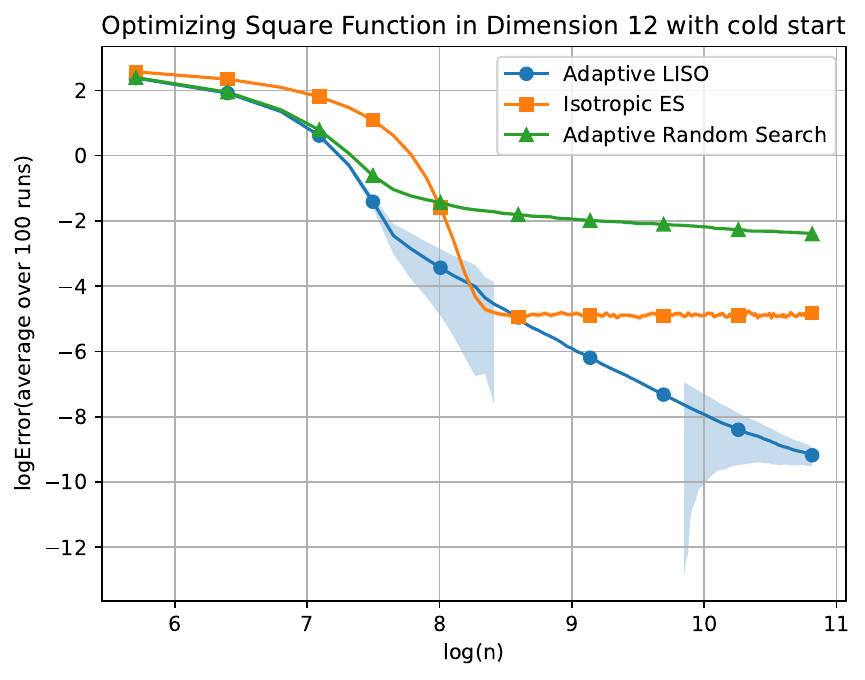}

    \end{subfigure}

    \caption{Function Square ($f_1$) in the adaptive case.}
\end{figure*}

\begin{figure*}[ht!]
    \centering
    \begin{subfigure}[b]{0.45\textwidth}
        \centering  
        \includegraphics[width=\textwidth]{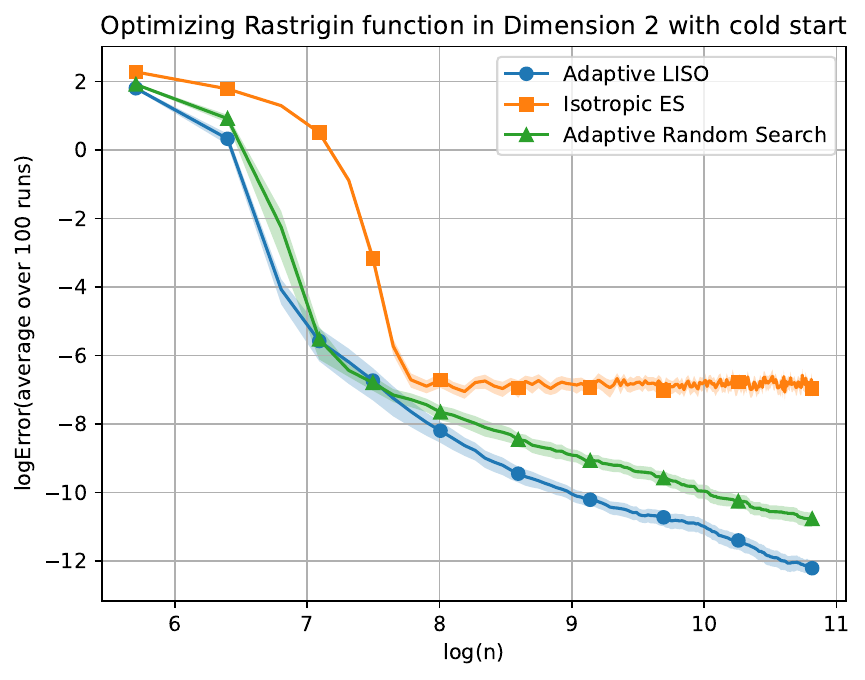}
    \end{subfigure}
    \hfill
    \begin{subfigure}[b]{0.45\textwidth}
        \centering
        \includegraphics[width=\textwidth]{Optimizing_Rastrigin_Function_in_Dimension_4_with_cold_start.pdf}
    \end{subfigure}

    \vskip\baselineskip

    \begin{subfigure}[b]{0.45\textwidth}
        \centering
        \includegraphics[width=\textwidth]{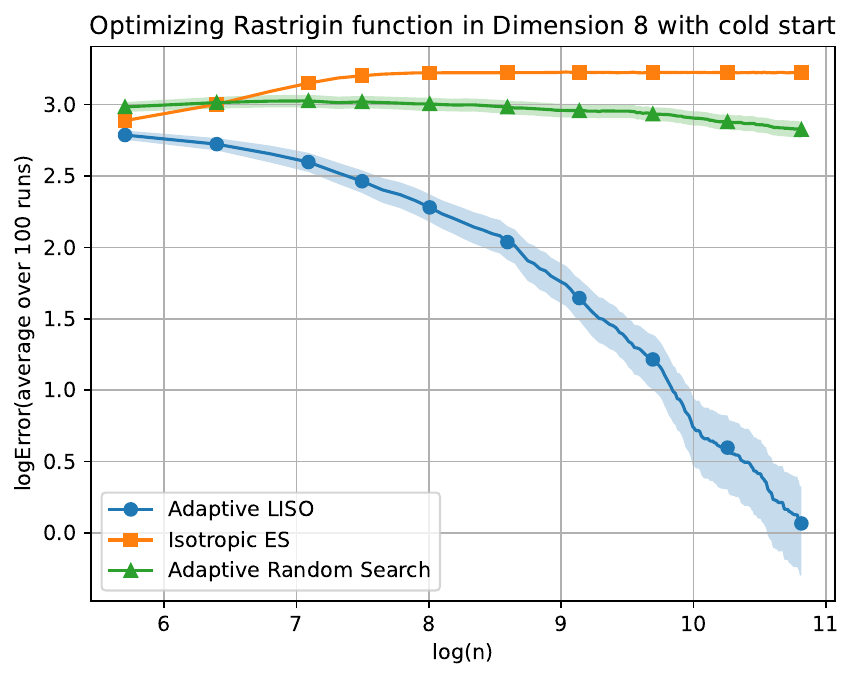}

    \end{subfigure}
    \hfill
    \begin{subfigure}[b]{0.45\textwidth}
        \centering
        \includegraphics[width=\textwidth]{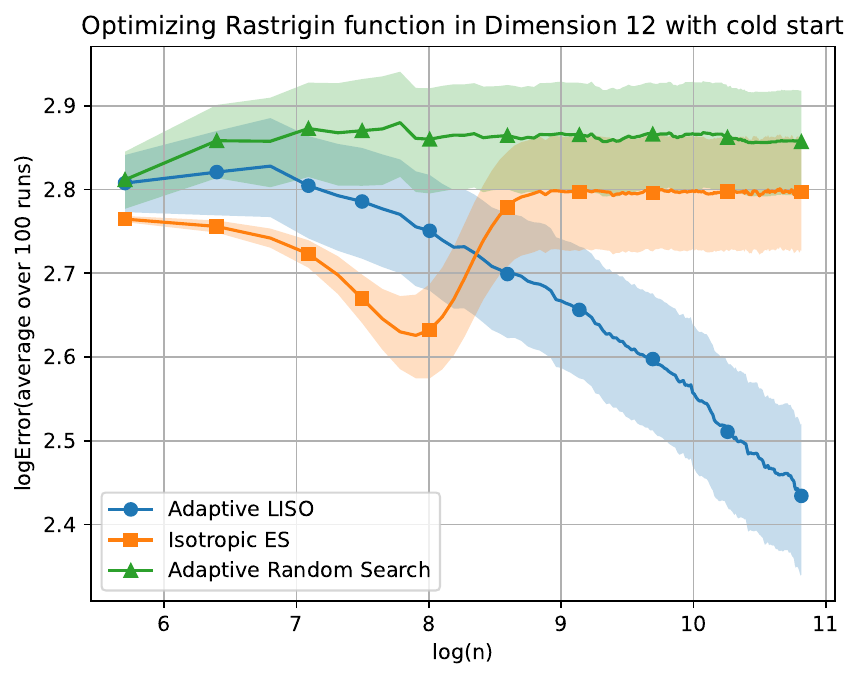}

    \end{subfigure}

    \caption{Function Rastrigin ($f_2$) in the adaptive case.}
\end{figure*}

\begin{figure*}[ht!]
    \centering

    \begin{subfigure}[b]{0.45\textwidth}
        \centering  
        \includegraphics[width=\textwidth]{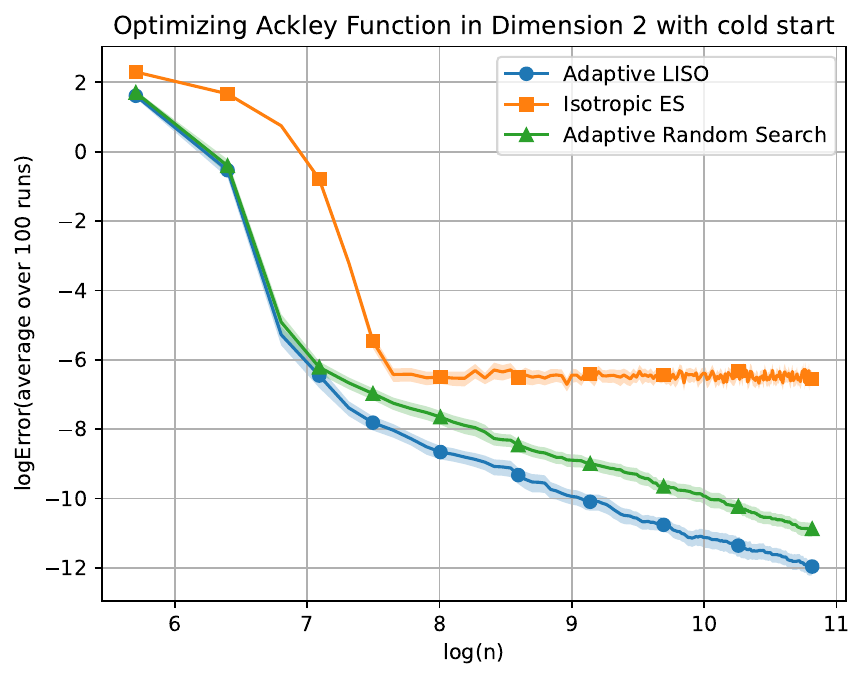}

    \end{subfigure}
    \hfill
    \begin{subfigure}[b]{0.45\textwidth}
        \centering
        \includegraphics[width=\textwidth]{Optimizing_Ackley_Function_in_Dimension_4_with_cold_start.pdf}

    \end{subfigure}

    \vskip\baselineskip

    \begin{subfigure}[b]{0.45\textwidth}
        \centering
        \includegraphics[width=\textwidth]{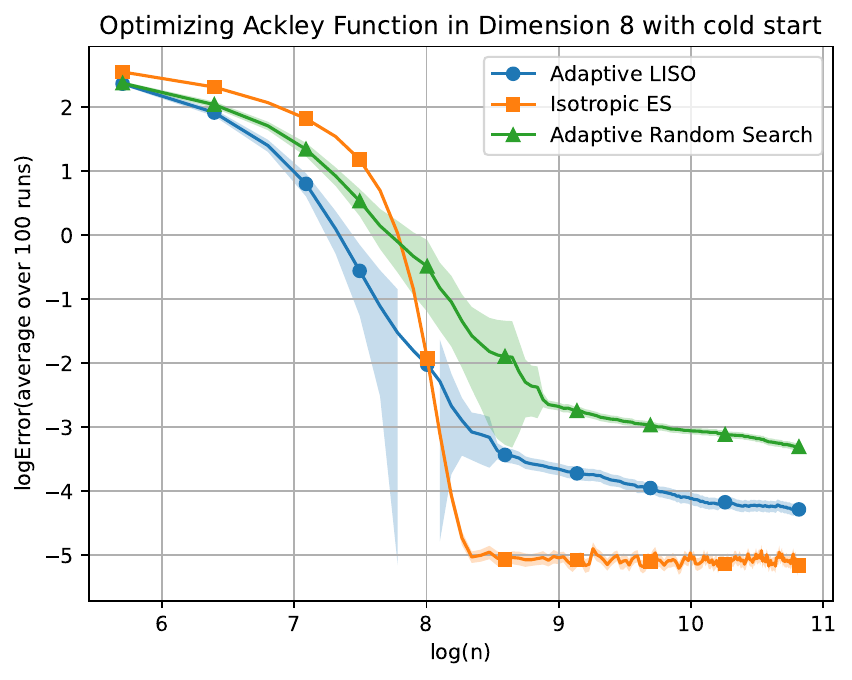}

    \end{subfigure}
    \hfill
    \begin{subfigure}[b]{0.45\textwidth}
        \centering
        \includegraphics[width=\textwidth]{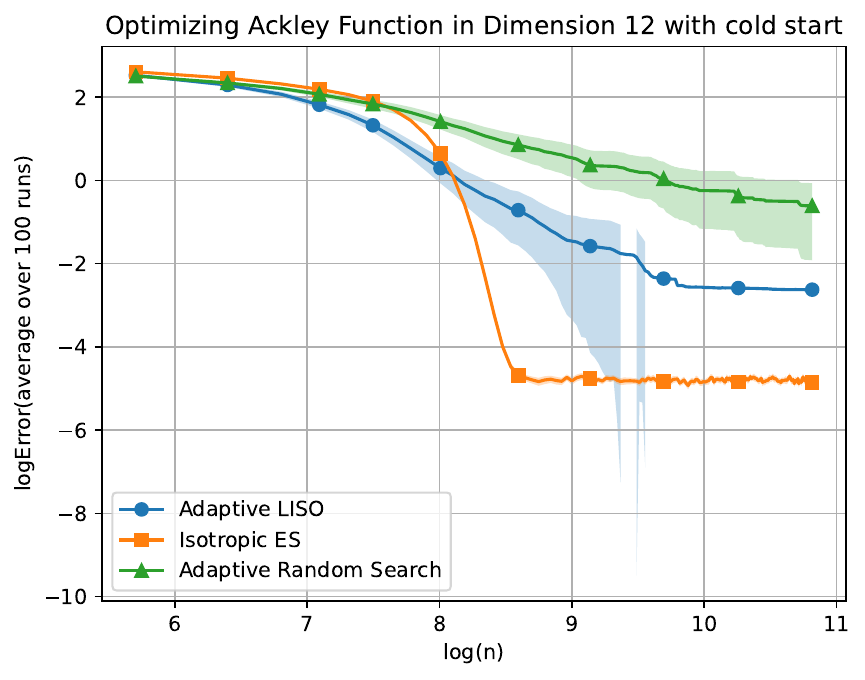}

    \end{subfigure}

    \caption{Function Ackley ($f_3$) in the adaptive case.}
\end{figure*}

\end{document}